\documentclass[11pt,reqno]{amsart}

\usepackage[pdfauthor={Dimitrios Chatzakos},
pdftitle={},
          pdfkeywords={Quantum variance},
           pdfcreator={Pdflatex}]
{hyperref}
\hypersetup{colorlinks=true}
\usepackage{xcolor}
\usepackage{graphicx}
\usepackage{enumerate}
\usepackage{amssymb,amsmath,latexsym,amsthm, mathtools}
\usepackage[english]{babel}

\newtheorem{theorem}{Theorem}[section]
\newtheorem{lemma}[theorem]{Lemma}
\newtheorem{proposition}[theorem]{Proposition}

\theoremstyle{definition}

\numberwithin{equation}{section}
\DeclareMathOperator*{\vol}{vol}

\DeclareMathOperator{\Sym}{sym}

\begin{document}

\newtheorem{remark}[theorem]{Remark}



\def \g {{\gamma}}
\def \G {{\Gamma}}
\def \l {{\lambda}}
\def \a {{\alpha}}
\def \b {{\beta}}
\def \f {{\phi}}
\def \r {{\rho}}
\def \R {{\mathbb R}}
\def \H {{\mathbb H}}
\def \N {{\mathbb N}}
\def \C {{\mathbb C}}
\def \Z {{\mathbb Z}}
\def \F {{\Phi}}
\def \Q {{\mathbb Q}}
\def \e {{\epsilon }}
\def \ev {{\vec\epsilon}}
\def \ov {{\vec{0}}}
\def \GinfmodG {{\Gamma_{\!\!\infty}\!\!\setminus\!\Gamma}}
\def \GmodH {{\Gamma\setminus\H^2}}
\def \sl  {{\textup{SL}_2( {\mathbb R})} }
\def \slc  { {\textup{PSL}_2({\mathbb C})}   }
\def \psl  {{\textup{PSL}_2( {\mathbb R})} }
\def \slz  {{\textup{SL}_2( {\mathbb Z})} }
\def \pslzi {{\textup{PSL}_2\left({\mathbb Z}[i]\right)} }
\def \pslz  {{\textup{PSL}_2( {\mathbb Z})} }
\def \L  {{\textup{L}^2}}
\def \GmodHthree {{\Gamma\setminus\H^3}}
\def \GmodHn {{\Gamma\setminus\H^n}}
\def \B {{\mathcal B}}
\newcommand{\rn}{\textup{reg}} 

\newcommand{\norm}[1]{\left\lVert #1 \right\rVert}
\newcommand{\abs}[1]{\left\lvert #1 \right\rvert}
\newcommand{\modsym}[2]{\left \langle #1,#2 \right\rangle}
\newcommand{\inprod}[2]{\left \langle #1,#2 \right\rangle}
\newcommand{\Nz}[1]{\left\lVert #1 \right\rVert_z}
\newcommand{\tr}[1]{\operatorname{tr}\left( #1 \right)}

\title[Quantum ergodicity for shrinking balls in hyperbolic manifolds]{Quantum 
ergodicity for shrinking balls in arithmetic hyperbolic manifolds}
\author{Dimitrios Chatzakos, Robin Frot and Nicole Raulf}
\address{Institut de Math\'ematiques de Bordeaux, Universit\'e de Bordeaux, B\^atiment A33, 
33405 Talence, France}
\email{dimitrios.chatzakos@math.u-bordeaux.fr}
\address{Laboratoire Painlev\'e LABEX-CEMPI, Universit\'e de Lille, 59655 Villeneuve
d'Ascq Cedex, France}
\email{robin.frot@univ-lille.fr}
\address{Laboratoire Painlev\'e LABEX-CEMPI, Universit\'e de Lille, 59655 Villeneuve
d'Ascq Cedex, France}
\email{nicole.raulf@univ-lille.fr}
\date{\today}

\begin{abstract}
We study a refinement of the quantum unique ergodicity conjecture for 
shrinking balls on arithmetic hyperbolic manifolds, with a focus on 
dimensions $ 2 $ and $ 3 $. For the Eisenstein series for the modular 
surface $ \pslz \setminus \mathbb{H}^2 $ we prove failure of quantum 
unique ergodicity close to the Planck-scale and an improved bound for 
its quantum variance. 

\noindent 
For arithmetic $ 3 $-manifolds we show that quantum unique ergodicity 
of Hecke--Maa{\ss} forms fails on shrinking balls centred on an 
arithmetic point and radius $ R \asymp t_j^{-\delta} $ with $ \delta 
> 3/4 $. For $ \mathrm{PSL}_2(\mathcal{O}_K) \setminus \mathbb{H}^3 $ 
with $ \mathcal{O}_K $ being the ring of integers of an imaginary 
quadratic number field of class number one, we prove, conditionally 
on the Generalized Lindel\"of Hypothesis, that equidistribution holds 
for Hecke--Maa{ss} forms if $ \delta < 2/5 $. Furthermore, we prove 
that equidistribution holds unconditionally for the Eisenstein series 
if $ \delta < (1-2\theta)/(34+4\theta) $ where $ \theta $ is the exponent 
towards the Ramanujan--Petersson conjecture. For $ \mathrm{PSL}_2(\mathbb{Z}[i]) $ 
we improve the last exponent to $ \delta < (1-2\theta)/(27+2\theta) $. 
Studying mean Lindel\"of estimates for $ L $-functions of Hecke--Maa{\ss} 
forms we improve the last exponent on average to $ \delta < 2/5$. 

\noindent 
Finally, we study massive irregularities for Laplace eigenfunctions on 
$ n $-dimensional compact arithmetic hyperbolic manifolds for $ n \geq 
4 $. We observe that quantum unique ergodicity fails on shrinking balls 
of radii $ R \asymp t^{-\delta_n+\epsilon} $ away from the Planck-scale, 
with $ \delta_n = 5/(n+1) $ for $ n \geq 5 $. 
\end{abstract}

\maketitle

\section{Introduction} \label{Introduction}

\subsection{Quantum ergodicity and restriction theorems}  
A central question in quantum chaos concerns the statistical behaviour of eigenfunctions 
of the Laplace operator on Riemannian manifolds of negative curvature. 
Let $ \mathcal{M} $ be a compact Riemannian manifold. We denote the volume element 
of $ \mathcal{M} $ by $ dv $ and the Laplace--Beltrami operator on $ \mathcal{M} $ 
by $ \Delta_{\mathcal{M}} $. Furthermore, let $ (\phi_j)_{j \geq 0} $ be an 
$\L$-normalized sequence of Laplace eigenfunctions with corresponding eigenvalues 
$ (\lambda_j)_{j \geq 0} $ which we order such that \mbox{$ \lambda_j \rightarrow 
\infty $} as $ j \rightarrow \infty $. To each of these eigenfunctions we can 
associate a probability measure via $ dv_j := |\phi_j|^2 dv $. The quantum ergodicity 
(QE) theorem of Schnirelman \cite{schnirelman}, Colin de Verdi\`ere \cite{deverdiere} 
and Zelditch \cite{zelditch} asserts that, if the geodesic flow on the unit cotangent 
bundle $ S^*\mathcal{M} $ is ergodic, then there exists a density one subsequence 
$ (\lambda_{j_k})_{k \geq 0} $ of $ (\lambda_j)_{j \geq 0} $ such that the 
corresponding measures $ dv_{j_k} $ converge weakly to the normalized measure 
$ \tfrac{1}{\vol(\mathcal{M})} \, dv $ as $ \lambda_{j_k} \rightarrow \infty $, 
i.e. 
\begin{align} \label{smallscalequantumergodicity}
\frac{1}{\vol(B)}\int_{B} |\phi_{j_k}|^2 dv \to \frac{1}{\vol(\mathcal{M})}
\end{align}
as $ k \rightarrow \infty $ for every continuity set $ B \subset \mathcal{M} $. 
In particular, the quantum ergodicity theorem holds for compact manifolds of 
negative curvature. Zelditch extended this result to the case of non-compact 
hyperbolic surfaces of finite volume \cite{Zelditch1992}. Furthermore, in 
\cite{zelditch2} he estimated the order of growth of sums of the form 
\begin{align*} 
S(\Lambda) := \sum_{\lambda_j \leq \Lambda} \left|\frac{1}{\vol(B)}\int_{B} 
|\phi_j|^2 dv - \frac{1}{\vol(\mathcal{M})}\right|^2,  
\end{align*}
$ B $ being fixed, and proved the nontrivial bound 
\begin{align} \label{variance1}
S(\Lambda) = o_{B}(\mathcal{N}(\Lambda)) 
\end{align} 
for the rate of quantum ergodicity. Here 
$ \mathcal{N}(\Lambda) = \# \{\lambda_j:\ \lambda_j \leq \Lambda\} $.  
Refinements of the ergodicity theorem and related conjectures have
been studied in various modifications. We mention the relation of 
quantum ergodicity to the random wave conjecture of Berry~\cite{berry} 
(see also \cite{hejhal}) which states that eigenfunctions of a classically 
ergodic system should show Gaussian random behaviour as the eigenvalues 
goes to infinity, i.e.\ the eigenfunctions should behave as random waves. 
Some of the most interesting refinements of the ergodicity theorem are 
the so-called restriction theorems. Here we are interested in the question 
whether the quantum unique ergodicity theorem, i.e.
(\ref{smallscalequantumergodicity}), still holds if the set $ B $ is
replaced by a sequence of sets whose size is decaying fast.  This
problem has been extensively studied in various cases, e.g.\ for the
$ n $-dimensional sphere $\mathbb{S}^n$ \cite{han2}, the $ n $-dimensional 
flat torus $\mathbb{T}^n = \mathbb{R}^n /2 \pi \mathbb{Z}^n$
\cite{hezaririviere1}, \cite{lesterrudnick}, \cite{han2} and
\cite{granvillewigman}, as well as for general compact manifolds of
negative sectional curvature \cite{han}, \cite{hezaririviere2} (see
subsection \ref{subsectquemodular}). In the next subsections we will
discuss the work of Lester and Rudnick \cite{lesterrudnick} in more
detail. In this paper we study quantum ergodicity on shrinking subsets 
for various arithmetic hyperbolic manifolds.

\subsection{Quantum unique ergodicity (QUE) for hyperbolic surfaces 
and shrinking balls} \label{subsectquemodular} 

For hyperbolic manifolds Rudnick and Sarnak \cite{rudnicksarnak} 
conjectured that (\ref{smallscalequantumergodicity}) holds for all 
eigenvalues, i.e.\ that 
\begin{align} \label{questatement} 
\frac{1}{\vol(B)} \int_{B} |\phi_j(z)|^2 dv(z)  
\longrightarrow \frac{1}{\vol(\mathcal{M})} 
\end{align}
as $ \lambda_j \to \infty $ for any fixed continuity set $ B $ of 
$\mathcal{M}$. 
This very deep and strong prediction is called the quantum unique 
ergodicity (QUE) conjecture. Using ergodic methods 
Lindenstrauss~\cite{lindenstrauss} was able to prove the conjecture 
in the case that the eigenfunctions $ (\phi_j)_{j \geq 0} $ are 
Hecke--Maa{\ss} forms on compact hyperbolic surfaces of a certain 
arithmetic type. In the case of not co-compact but co-finite 
hyperbolic surfaces he was only able to determine the quantum limit 
up to a constant $ c $. Later Soundararajan~\cite{soundararajan} 
proved that $ c = 1 $ so that the quantum unique ergodicity conjecture 
is now known for Hecke--Maa{\ss} forms on hyperbolic surfaces of arithmetic 
type. Furthermore, Silberman and Venkatesh \cite{silbermanvenkatesh1}, 
\cite{silbermanvenkatesh2} established the QUE conjecture for compact 
quotients of higher rank real Lie groups. For general hyperbolic manifolds 
the quantum unique ergodicity conjecture is still open. 

We now turn to the case of the modular group $ \Gamma := \pslz $ which
acts on the hyperbolic plane $ \mathbb{H}^2 = \{x+iy \in \C:\ y > 0\} $ 
by linear fractional transformations. As usual we equip $ \mathbb{H}^2 $ 
with the hyperbolic metric and denote the corresponding volume element 
by $ d\mu(z) $. The corresponding Laplace operator is 
\begin{equation*}
\Delta = 
y^2 \left(\frac{\partial^2}{\partial x^2} 
+ \frac{\partial^2}{\partial y^2}\right).
\end{equation*}
Furthermore, let $ (u_j)_{j \geq 0} $ be an orthonormalized basis of 
eigenfunctions of $ -\Delta $ with the property that they are 
simultaneous eigenfunctions of the Hecke operators $ T_n $, $ n \in 
\N $. We write the eigenvalue $ \lambda_j $ of $ u_j $ as $ \lambda_j 
= 1/4 + t_j^2 $ where we choose $ \Re(t_j) \geq 0 $. Luo and 
Sarnak~\cite{luo_quantum_1995} studied refined equidistribution results 
for automorphic forms on the modular surface $ \mathcal{M}_{\Gamma} := 
\Gamma \setminus \mathbb{H}^2 $, point-wise for the Eisenstein series 
and on average for Hecke--Maa{\ss} forms. For the quantum variance of 
Hecke--Maa\ss\ forms they proved the upper bound
\begin{equation} \label{variance2}
\sum_{|t_j| \leq T} \left|\frac{1}{\vol(B)}\int_{B} |u_j(z)|^2 d\mu (z) 
- \frac{1}{\vol(\GmodH)} \right|^2 
= O_{B, \epsilon}\left( T^{1+\epsilon}\right)
\end{equation}
for every fixed continuity set $ B $ (see \cite[Theorem~1.4]{luo_quantum_1995}). 
Since the Weyl law implies that 
$  
\mathcal{N}(T) := \#\{j \geq 0:\ |t_j| \leq T\} \sim T^2/12 
$,  
this is a square-root improvement of Zelditch's bound (\ref{variance1}). 
It also implies the bound 
\begin{align} \label{optimalrate}
\frac{1}{\vol(B)}\int_{B} |u_j(z)|^2 d \mu (z) 
= \frac{1}{\vol(\GmodH)} + O_{B, \epsilon}\left(t_j^{-1/2+\epsilon}\right)
\end{align}
on average. Luo and Sarnak conjectured that this rate of convergence 
holds point-wise. This is supported by Watson's triple product formula 
(see Subsections~\ref{watsonformuladimension2} and 
\ref{watsonichinoformuladimension3}) and if true it is optimal. For the 
Eisenstein series $ E(z, s) $ which is formally an eigenfunction of the 
Laplace operator with corresponding eigenvalue $ s(1-s) = 1/4 + t^2 \geq 1/4$ and 
the Hecke operators but are not in $ \L(\GmodH)$. Luo and Sarnak~\cite{luo_quantum_1995} 
proved the asymptotic behaviour 
\begin{align} \label{questatementeisenstein}
\frac{1}{\log\left(\frac{1}{4}+t^2\right) \vol(B)} 
\int_{B} \left|E\left(z,\frac{1}{2}+it\right)\right|^2 d\mu(z)  
= \frac{1}{\vol(\GmodH)} + O_{B}\left(\frac{1}{\log \log t} \right)
\end{align}
as $ t \rightarrow \infty $. Luo and Sarnak's upper bound in (\ref{variance2}) was refined by 
Zhao~\cite{zhao} and Sarnak--Zhao~\cite{sarnakzhao} while Huang~\cite{huang} 
derived an asymptotic formula for the quantum variance of Eisenstein series 
for PSL$_2(\Z) $. 

In the situation of the modular group restriction problems now ask whether 
QUE, i.e.\ (\ref{questatement}) and (\ref{questatementeisenstein}), still 
hold if the fixed set $ B $ is replaced by a sequence of balls $ B_R(w) 
\subset \mathcal{M}_{\Gamma} $ whose centre $ w \in \mathcal{M}_{\Gamma} $ 
is fixed and whose radii $ R = R_j \rightarrow 0 $ as $ t_j \rightarrow 
\infty $. 
These small scale equidistribution problems can be understood as an 
alternative way to quantify the rate of convergence in (\ref{questatement}). 
Physical heuristics indicate that we cannot expect equidistribution below 
the Planck-scale $1/ \sqrt{\lambda_j} $ (also called de-Broglie wavelength) as in this 
range quantum phenomena disappear and the Laplace eigenfunctions behave 
like regular functions. For the modular surface this would imply that 
(\ref{questatement}) and (\ref{questatementeisenstein}) do not hold anymore 
if $ R \asymp t_j^{-1} \asymp \lambda_j^{-1/2} $. However, Berry's random 
wave conjecture~\cite{berry} (see also Hejhal--Rackner~\cite{hejhal} and 
Lester--Rudnick~\cite{lesterrudnick}) implies that one should expect QE 
or even QUE close to the Planck-scale, i.e.\ that (\ref{questatement}) and 
(\ref{questatementeisenstein}) hold if 
\begin{equation*} \label{planckscale}
R \gg_{\epsilon} t_j^{-1+\epsilon}
\end{equation*}
for any $\epsilon >0$. Investigating the topography of Hecke--Maa{\ss} cusp forms 
on the modular surface Hejhal and Rackner~\cite{hejhal} provided evidence that 
supports QUE up to the Planck-scale (\ref{planckscale}). 
Luo and Sarnak \cite{luo_quantum_1995} were the first to prove that quantum 
ergodicity holds for shrinking balls on the modular surface if their radii 
satisfy $ R \gg_{\epsilon} t_j^{-\delta + \epsilon}$ for some small fixed $ \delta > 0 $.
Following Sarnak's letter to Reznikov~\cite{sarnak2}, Young \cite{young} 
and Humphries \cite{humphries} studied the QUE problem on thin sets of 
$ \pslz \setminus \mathbb{H}^2$. Young investigated ergodicity of 
Hecke--Maa{\ss} cusp forms and Eisenstein series on infinite geodesics 
and on shrinking balls of the modular surface. Applying product formulae 
he obtained that QUE for shrinking balls holds if the rate of decay 
satisfies $ R \gg t_j^{-\delta + \epsilon} $ for some fixed $ \delta > 0 $. 
Since he refers to product formulae his result for Hecke--Maa{\ss} cusp 
forms is conditional on the Generalized Lindel\"of Hypothesis (GLH) 
(see \cite[Prop.~1.5]{young}) whereas for Eisenstein series his result 
is unconditional (see \cite[Thm.~1.4]{young}). However, for the Eisenstein 
series the error term is worse than the one given for Hecke--Maa{\ss} cusp 
forms. Recently, Humphries improved Young's result for Eisenstein series 
(see \cite[Thm.~1.16]{humphries}). Furthermore, he proved that
equidistribution for Hecke--Maa{\ss} cusp forms fails close to 
the Planck-scale (see \cite[Thm.~1.14]{humphries}). 
We summarize these results in the following theorems: 

\begin{theorem}[Young \cite{young}, Humphries \cite{humphries}] \label{younghumphries}
Let $ w $ be a fixed point on $ \textup{PSL}_2(\Z) \setminus \mathbb{H}^2 $ 
and $ B_{R}(w) $ a ball of centre $ w $ and radius $ R $ where $ R \rightarrow 0 $ 
as $ t_j \rightarrow \infty $. Then for any $ \epsilon > 0 $: 

\smallskip 

\begin{enumerate}[(a)]  
\item 
Let $ (u_j)_j $ be a sequence of $ \textup{L}^2 $-normalized Hecke--Maa{\ss} forms 
and $ R \gg t_{j}^{-\delta +\epsilon}$ with $\delta \leq 1/3 $. Assuming GLH we 
have, as $ t_j \rightarrow \infty $, 
\begin{align}
\frac{1}{\vol(B_{R}(w))}\int_{B_{R}(w)} |u_j(z)|^2 d\mu(z) 
= \frac{1}{\vol(\GmodH)} + o_{w, \delta} (1). 
\end{align}
\item \label{part_b} 
Let $ R \gg t^{-\delta+\epsilon} $ with $ \delta \leq 1/6 $. For the Eisenstein 
we have unconditionally, as $ t \rightarrow \infty $,  
\begin{align}
\frac{1}{\log(\frac{1}{4}+t^2) 
\vol(B_{R}(w))}\int_{B_{R}(w)} \left|E\left(z,\frac{1}{2}+it\right)\right|^2 d\mu(z)  
= \frac{1}{\vol(\GmodH)} + o_{w, \delta} (1).
\end{align}
\end{enumerate} 
\end{theorem}

\begin{theorem}[Humphries \cite{humphries}] \label{humphriesomegatheorem} 
Let $ w $ be a fixed Heegner point on $ \textup{PSL}_2(\Z) \setminus 
\mathbb{H}^2 $ and $ f(t) $ a function satisfying $ \underset{t \to 
\infty}{\lim} f(t) \rightarrow \infty $ and 
\begin{align} 
f(t) = 
o\left( \exp \left(2 \sqrt{\frac{\log t}{\log \log t} \left(1+ 
O\left(\frac{\log \log \log t}{\log \log t} \right)\right)} \right)\right).
\end{align}
Then for $ R \ll t_j^{-1} f(t_j) $ we have
\begin{align}
\frac{1}{\vol(B_{R}(w))}\int_{B_{R}(w)} |u_j(z)|^2 d\mu(z) 
= \Omega \left( \frac{|u_j(w)|^2}{(f(t_j))^{3/2}} \right) \neq O(1). 
\end{align}
In particular, quantum unique ergodicity (\ref{questatement}) fails in this range. 
\end{theorem}

\begin{remark}  
In particular, Theorem~\ref{humphriesomegatheorem} implies the failure 
of equidistribution in the range $ R \ll t_j^{-1} (\log t_j)^a $ for 
any $ a > 0 $. This can be compared to the case of the Euclidean torus 
$ \mathbb{T}^2 $ where Granville and Wigman~\cite{granvillewigman} showed 
that quantum ergodicity holds for $ R \gg t_j^{-1} (\log t_j)^{1 + 
\frac{\log 2}{3}} $ but fails for $ R \ll t_j^{-1} (\log t_j)^{\frac{\log 2}{2}} $. 
\end{remark} 

Moreover, Humphries also obtained spatial variance bounds over the 
surface that support (\ref{planckscale}) for almost all points $ w 
\in \pslz \setminus \mathbb{H}^2 $. In the case of Hecke--Maa{\ss} 
cusp forms his results are conditional on the Generalized Lindel\"of 
Hypothesis whereas the results are unconditional in the Eisenstein 
series case (see \cite[Thm.~1.17, 1.18]{humphries}). 
Furthermore, Humphries and Khan \cite{humphrieskhan} proved that 
small scale mass equidistribution holds all the way down to the 
Planck-scale $ R \gg t_j^{-1 + \epsilon} $ on surfaces generated 
by special Hecke congruence subgroups when considering only the 
sparse subsequence of dihedral Hecke--Maa{\ss} forms. 
For general $ n $-dimensional negatively curved manifolds Han~\cite{han} 
and Hezari and Rivi\`ere~\cite{hezaririviere2} proved, using ergodic tools, 
that quantum ergodicity holds for shrinking balls $ B_{R} $ if 
\begin{equation} \label{hanhezriv}
R \gg (\log t_j)^{-\frac{1}{2n}+\epsilon}. 
\end{equation}
That means in the general case we know ergodicity only in shrinking balls of 
very slow (logarithmic) decay.  

\subsection{Failure of QUE for Eisenstein series close to the Planck-scale}

We first prove an analogue of Theorem \ref{humphriesomegatheorem} for 
Eisenstein series.

\begin{theorem}\label{omegatheoremeisenstein} 
Let $ w $ be a fixed Heegner point on the modular surface. There exists 
a constant $ C = C(w) $ such that if $ f(t) $ is a function satisfying 
\mbox{$ \underset{t \to \infty}{\lim} f(t) \rightarrow 
\infty $} and
\begin{equation} \label{upperboundforf}
f(t) 
= 
o\left(\exp\left(C\sqrt{\frac{\log t}{\log\log t}}\right) (\log t)^{-7/9}\right),
\end{equation}
we have for $ R \ll t^{-1} f(t) $ 
\begin{equation*}
\frac{1}{\log \left(\frac{1}{4}+t^2\right) \vol(B_{R}(w))} \int_{B_{R}(w)} 
\left|E\left(z,\frac{1}{2}+it\right)\right|^2 d\mu(z) \neq O(1). 
\end{equation*}

\noindent 
In particular, quantum unique ergodicity for Eisenstein series fails 
if $ R \ll t^{-1} f(t) $.
\end{theorem}

Therefore we see that, as in the case of Maa{\ss} cusp forms, equidistribution 
of Eisenstein series fails for radii $ R \ll t^{-1} (\log t)^a $ for any $ a 
> 0 $.

\subsection{Quantum variance of Eisenstein series on shrinking balls of 
the modular surface}

Our second result is a uniform upper bound for the quantum variance of 
Eisenstein series on shrinking balls of the modular surface. As in Young 
\cite{young} we apply the spectral theorem and product formulae to 
prove a uniform estimate in $ R $ and $ T $. This is the shrinking balls 
analogue of the Luo-Sarnak bound (\ref{variance2}) and Huang~\cite{huang}. 
We improve on average the exponents obtained by Young and 
Humphries (see Part~(\ref{part_b}) of Theorem~\ref{younghumphries}) and 
prove the following result: 

\begin{theorem} \label{theoremeisend2}
Let $ w $ be a fixed point on $ \textup{PSL}_2(\Z) \setminus \mathbb{H}^2 $. 
The quantum variance of Eisenstein series satisfies the uniform upper bound 
\begin{align} \label{int116}
\int_{T}^{2T} 
\left|\frac{\int_{B_R(w)}|E(z,1/2+it)|^2 d\mu(z)}{\log(\frac{1}{4}+t^2) \vol(B_R(w))} 
- \frac{1}{\vol(\GmodH)}\right|^2 dt 
\ll_{w} 
\frac{T^{\epsilon}}{R^{3+\epsilon}} + \frac{T^{\epsilon}}{R^{2+\epsilon}} 
+ \frac{T}{(\log \log T)^2}. 
\end{align}

\noindent 
In particular, quantum ergodicity holds for $ R \gg t^{-\delta+\epsilon} $ 
with $ \delta \leq 1/3 $. 
\end{theorem}

Note that the first term in (\ref{int116}) comes from the contribution of 
the discrete spectrum. The second term, which is strictly smaller than the 
first one, comes from the contribution of the continuous spectrum. The third 
term is independent of $R$ and comes from the degenerate contribution in the 
generalized Plancherel formula.

\subsection{Spectral theory and QUE on arithmetic hyperbolic 3-manifolds}

The main part of this paper focuses on the case of $ 3 $-dimensional 
arithmetic manifolds. Before stating our results we briefly introduce 
the notation used in this paper as well as the most important results 
for the spectral theory of automorphic forms on hyperbolic $ 3 $-space 
following \cite{elstrodt} and their notation. 
Let $ \mathbb{H}^3 := \{P = z + rj:\ z \in \C, \ j > 0\} $ be the 
hyperbolic $ 3 $-space which we consider as a subset of Hamilton's 
quaternions with the standard basis $ 1, \ i, \ j $ and $ k $. As 
usual $ \mathbb{H}^3 $ is equipped with the hyperbolic metric 
and we denote the corresponding volume element by 
\begin{align} \label{hypmeasure}
dv = dv(P) = \frac{dx dy dr}{r^3} 
\end{align}
and the corresponding Laplace-Beltrami operator by $ \Delta $, 
\begin{align}\label{hypoperator}
\Delta = 
r^2 \left(\frac{\partial^2}{\partial x^2} + \frac{\partial^2}{\partial y^2} 
+ \frac{\partial^2}{\partial r^2} \right) - r \frac{\partial}{\partial r}.
\end{align}
The group PSL$_2(\C) $ acts on $ \mathbb{H}^3 $ in a natural way: let 
$ 
M 
= 
\left(\begin{smallmatrix}a&b\\c&d\end{smallmatrix}\right) 
$ 
be an element of 
$ 
\mathrm{PSL}_2(\mathbb{C}) 
$. 
Then, if we understand $ P \in \mathbb{H}^3 $ as a quaternion whose 
fourth component is zero, $ MP $ is given by 
\begin{align*}
M P :=(aP+b)(cP+d)^{-1}.
\end{align*}
Here this inverse is taken in the skew field of quaternions. This 
action is orientation-preserving isometric. 
Let $ K = \mathbb{Q}(\sqrt{D}) $, $ D < 0 $, be an imaginary quadratic 
number field of class number $ H_K = 1 $. There are nine such imaginary 
quadratic fields, namely for $ D= -1, -2, -3, -7, -11, -19, -43 $, 
$ -67 $ and $ -163 $. We denote the ring of integers of $ K $ by 
$ \mathcal{O}_K $ and its unit group by $ \mathcal{O}_K^{*} $. If 
$ D = -1 $ we have $ \mathcal{O}_K^{*} = \{\pm 1, \ \pm i\} $, if 
$ D = -3 $ then $ \mathcal{O}_K^{*} = \{\pm 1, \pm \rho , \pm \rho^2\} $, 
$ \rho := \tfrac{1}{2} (-1+ i\sqrt{3}) $, otherwise $ \mathcal{O}_K^{*} 
= \{\pm 1\} $. 
Furthermore, we denote the discriminant of $ K $ by $ d_K $ 
and $ N(n) := |n|^2 $ is the norm of $ n \in \mathcal{O}_K $. 
The ring $ \mathcal{O}_K $ can be viewed as a lattice in $ \mathbb{R}^2 $ with 
fundamental parallelogram $F \subset \mathbb{R}^2$. 
Apart from co-compact groups $ \Gamma \subset \textup{PSL}_2(\C) $ we 
are interested in the so-called Bianchi groups $ \Gamma := 
\textup{PSL}_2(\mathcal{O}_K) $. These are co-finite subgroups of 
PSL$_2(\C) $. As we only work with imaginary quadratic number fields 
whose class number is one, we know that up to $ \Gamma $-equivalence 
the group $ \Gamma $ has only the cusp $ \infty $. The spectral theory 
of $ -\Delta $ on $ \mathcal{M}_{\Gamma} = \Gamma \setminus \mathbb{H}^3 $ 
is well-known (see e.g.\ \cite{elstrodt}). As $ \Gamma $ is not co-compact 
but co-finite the spectrum of $ - \Delta $ consists of a discrete part 
containing the eigenvalues $ \lambda_j = 1 + t_j^2 $ and an absolutely 
continuous part spanning $ [1,\infty) $ with multiplicity $ 1 $. The 
absolutely continuous part is given by the Eisenstein series. Let 
$ \Gamma_{\infty} = \{\gamma \in \Gamma:\ \gamma \infty = \infty\} $ 
be the stabilizer of the cusp $ \infty $. Note that for $ D = -1 $ 
and $ D = -3 $ the stabilizer also contains elliptic elements of 
$ \Gamma $. Therefore, we define 
$ 
\Gamma'_{\infty} = 
\{\gamma \in \Gamma_{\infty}: |\textup{tr} \, \gamma| = 2\} 
$ 
to be the set of all parabolic elements of $ \Gamma $ that stabilize 
$ \infty $. Then the Eisenstein series is given by 
\begin{equation*} 
E(P, s) 
= \sum_{\gamma \in  \Gamma'_{\infty} \setminus \Gamma} r(\gamma P)^s, \quad  \Re(s) > 2 .
\end{equation*} 
The theory of Eisenstein series for the hyperbolic 
space as we need it in this paper can be found in \cite{elstrodt} 
or \cite{heitkamp}. 
Note that we normalize the Eisenstein series so that the critical 
line is $ \Re(s) = 1 $ as it is also done in \cite{laaksonen}. A 
different way to define the Eisenstein series is 
\begin{equation*} 
E_{\infty}(P, s) 
= \sum_{\gamma \in  \Gamma_{\infty} \setminus \Gamma} r(\gamma P)^s, \quad \Re(s) > 2 . 
\end{equation*} 
Then $ E(P, s) = \frac{|\mathcal{O}_K^*|}{2} E_{\infty}(P, s) $ (see 
e.g.\ \cite{elstrodt}, p.~232). In order to  give the Fourier expansion 
of the Eisenstein series let $ \zeta_K(s) $ denote the Dedekind 
zeta function for $ K $ and 
\begin{equation} \label{eq:scattering_H3} 
\phi(s) := \frac{2 \pi}{s \sqrt{|d_K|}} \frac{\zeta_K(s)}{\zeta_K(1+s)} 
\end{equation} 
be the scattering matrix for $ \textup{PSL}_2(\mathcal{O}_K) $. Then 
the Fourier expansion of $ E_{\infty}(P, 1+s) $ is given by 
\begin{equation} \label{eq:FE_Eisenstein_H3} 
\begin{split} 
E_{\infty}(P, 1+s) &= 
r^{1+s} + \phi(s) r^{1-s} 
+ \frac{2 (2 \pi)^{1+s}}{|d_K|^{(1+s)/2} \Gamma(1+s) \zeta_K(1+s)} \\ 
& \quad \cdot 
\sum_{0 \not= \omega \in \mathcal{O}_K} |\omega|^s \sigma_{-s}(\omega) 
r K_s\left(\frac{4 \pi |\omega| r}{\sqrt{|d_K|}}\right) e^{2 \pi i 
\left\langle \frac{2 \overline{\omega}}{\sqrt{d_K}}, z\right\rangle} 
\end{split} 
\end{equation} 
(see e.g.\ \cite{elstrodt}, Theorem~2.11, pp.~369--370 or \cite{heitkamp}, 
p.~102). Here $ \sigma_s(\omega) $ denotes the generalized divisor 
function 
\begin{equation*} 
\sigma_s(\omega) = \frac{1}{|\mathcal{O}_K^{*}|} 
\sum_{\substack{d \in \mathcal{O}, \\ d \mid \omega}} 
|d|^{2s}. 
\end{equation*} 
In the three-dimensional case the Hecke operators are 
defined as follows: if $ n \in \mathcal{O} \setminus \{0\} $ 
we define $ \mathcal{M}_n $ to be the set of all matrices of 
the form 
$ 
\left(\begin{smallmatrix} a & b \\ c & d \end{smallmatrix}\right) 
$, 
$ ad - bc = 1 $. Then for $ f $ being a $ \Gamma $-invariant 
function the Hecke operator $ T_n $ is given by 
\begin{equation*} 
(T_n f)(P) := 
\frac{1}{\sqrt{N(n)}} 
\sum_{\gamma \in \Gamma \setminus \mathcal{M}_n} f(\gamma P). 
\end{equation*} 
The theory for Hecke operators for Bianchi groups is developed in 
\cite{heitkamp}. 
However, in contrast to Heitkamp we have incorporated the factor 
$ 1/\sqrt{N (n)} $ in the definition of the Hecke operator. 

\bigskip 

In the present work we initiate the study of quantum equidistribution 
results for shrinking balls on the $ 3 $-dimensional arithmetic 
hyperbolic manifolds $ \mathcal{M}_{\Gamma} = \Gamma \setminus \mathbb{H}^3$. 
Rudnick and Sarnak \cite{rudnicksarnak} conjectured that the sequence 
of eigenstate measures on higher dimensional manifolds still converges 
to the volume measure, i.e.\ (\ref{questatement}) still holds. 
However, they noticed that the random wave model (RWM) does not apply to
Laplace eigenfunctions on compact arithmetic $3$-manifolds since
the sup-norm of Laplace eigenfunctions can be significantly large 
in this case. In fact, the analogue of conjecture (\ref{questatement}) 
is still unknown in three dimensions as it does not follow from the work 
of Silberman and Venkatesh \cite{silbermanvenkatesh1},
\cite{silbermanvenkatesh2}). Quantum ergodicity of Eisenstein series 
for the hyperbolic $ 3 $-space was established by Petridis and Sarnak
\cite{petridissarnak} and Koyama \cite{koyama1} who studied related
subconvexity estimates for Rankin-Selberg convolutions.
Let $ u_j $ be an $ \L $-normalized Hecke--Maa{\ss} forms with corresponding 
eigenvalue $ 1 + t_j^2 $. As usual we can define a measure $ v_j $ on 
$ \mathcal{M} $ via 
\begin{align*}
dv_j := |u_j|^2 dv.
\end{align*}
The GLH for the $L$-functions appearing in the corresponding triple 
product formula (see Theorem~\ref{theoremtripleproductichinomarshall}) 
implies that the expected rate of convergence for a continuity set 
$ B\subset\mathcal{M} $ should be 
\begin{equation} \label{variance2threedimensions}
\frac{1}{\vol(B)} \int_{B} |u_j(Q)|^2 dv(Q) = 
\frac{1}{\vol(\GmodHthree)} + O_{B, \epsilon}( t_j^{-1+\epsilon}).
\end{equation}
Furthermore,  the quantum ergodicity result of Koyama for the Eisenstein 
series reads as follows: 
\begin{equation} \label{koyamamainterm}
\frac{1}{\log(1+t^2) \vol(B)} \int_{B} \left|E\left(Q,1+it\right)\right|^2 dv(Q) 
= \frac{|\mathcal{O}_K^{*}| \sqrt{|d_K|}}{4 \vol(\GmodHthree)} + O_{B, \Gamma}\left(\frac{1}{\log \log t} \right)
\end{equation}
as $t \to \infty$ for $K$ an imaginary quadratic number field of class number one. 
In fact, Koyama \cite[p.~485]{koyama1} derived a wrong main term in (\ref{koyamamainterm}) 
which was corrected by Laaksonen \cite[Rem.~1]{laaksonen}. 

\subsection{QUE on shrinking balls for $ 3 $-manifolds}

We study the QUE conjecture for shrinking balls $ B_{R}(P) \subset 
\mathcal{M}_{\Gamma} $ with fixed center $ P \in \mathcal{M}_{\Gamma} $ 
on arithmetic 3-manifolds $\mathcal{M}_{\Gamma} = \Gamma \setminus \mathbb{H}^3$. 
The volume of a hyperbolic ball in $ \mathbb{H}^3 $ of radius $ R $ and center 
$ P $ is given by (see \cite[Eq.~(2.6)]{elstrodt})
\begin{align} \label{volumesphere}
\vol(B_R (P)) = \pi (\sinh (2R) - 2R) \sim  R^3
\end{align}
as $ R \rightarrow 0 $. By the reasoning of Berry's conjecture one 
would also in this case expect quantum (unique) ergodicity close 
to the Planck-scale $ R \asymp t_j^{-1+\epsilon} $. However, our 
first result shows that for specific arithmetic hyperbolic 
$ 3 $-manifolds we cannot have equidistribution close to the 
Planck-scale for balls centred on arithmetic points. In the 
following proposition we use the notion of the QCM points and 
QCM-manifolds as defined by Mili\'cevi\'c \cite[p.~1380]{milicevic} 
as well as the notion of a manifold of Maclachlan--Reid type 
(see \cite[p.~1381]{milicevic}). These points play the role of 
the classical CM-points on Riemann surfaces. Indeed, for every 
QCM-point $ P \in \GmodHthree $ Mili\'cevi\'c proved that $ u_j $ 
admits large values at $ P $, thus obtaining his strong lower bound 
for the sup-norm problem (see \cite[Thm.~1]{milicevic}): 
\begin{equation}  \label{milicevicbound}
\left|u_j (P)\right| = 
\Omega\left(t_j^{\frac{1}{2}+ O\big(\frac{1}{\log \log t_j}\big)}\right).
\end{equation}

\begin{theorem} \label{proposition1} 
The following two statements hold: 

\medskip 

\noindent 
(a) 
Let $ \Gamma $ be an arithmetic co-finite Kleinian group and $ u_j $ be 
an $\L$-normalized Hecke--Maa\ss\ eigenform of $ \Gamma $. If $ \Gamma 
\setminus \mathbb{H}^3 $ is of Maclachlan--Reid type, $ P $ is a fixed 
QCM-point on $ \GmodHthree $ and $ R \ll t_j^{-\delta} $ with $ \delta > 3/4 $ fixed, then
\begin{equation*}
\frac{1}{\vol(B_R (P))} \int_{B_R(P)} |u_j(Q)|^2 dv(Q) 
= \Omega\left( \frac{|u_j(P)|^2}{ t_j^{4-4\delta}}\right) \neq O(1).
\end{equation*}

\medskip 

\noindent 
(b) 
There exist co-compact arithmetic groups $ \Gamma $ and arithmetic points 
$ P \in \GmodHthree $ such that if $ R \ll t_j^{-\delta} $ with 
$ \delta > 3/4 $ fixed, then
\begin{align}
\frac{1}{\vol(B_R (P))} \int_{B_R(P)} |u_j(Q)|^2 dv(Q) \neq O(1). 
\end{align}
In particular, in both cases (a) and (b) QUE fails on these shrinking balls, 
away from the Planck-scale. 
\end{theorem}

We clarify some points in the statement of Theorem~\ref{proposition1}. 
Our proof relies on the $\Omega$-results of Rudnick and Sarnak~\cite{rudnicksarnak} 
and Mili\'cevi\'c \cite{milicevic} for the sup-norm problem on arithmetic 
$ 3 $-manifolds (we are thus indirectly using the existence of explicit 
theta lifts). Part $ (a) $ of Theorem~\ref{proposition1} follows from 
(\ref{milicevicbound}). Similarly, case $(b)$ follows from the classical 
lower bound of Rudnick and Sarnak~\cite{rudnicksarnak} which was the first 
case of a hyperbolic manifold exhibited with large sup-norms of Laplace 
eigenfunctions. Although they are less explicit in identifying precise 
classes of manifolds with power growth of $ \|u_j\|_{\infty} $ their result 
holds for manifolds where one can construct explicit theta lifts orthogonal 
to sufficiently many Maa\ss\ forms. Other results of this type have been 
proved by Lapid and Offen \cite{offen} using their Waldspurger-type formula. 
We refer to \cite[Subsect.~0.5]{milicevic} and \cite[Subsect.~1.3]{brumley} 
for detailed discussions on this subject. 

\bigskip 

On the other hand, for co-finite arithmetic $ 3 $-manifolds with one cusp 
we prove that QUE holds for shrinking balls of some rate. It is remarkable 
that the study of explicit small scale equidistribution can be better 
understood in three dimensions than in two.

\begin{theorem} \label{theorem2}
Let $\GmodHthree$ be an arithmetic hyperbolic 3-manifold with $ \Gamma = 
\textup{PSL}_2(\mathcal{O}_K)$ being a Bianchi group of class number one 
and $ P \in \GmodHthree $ be a fixed point. Then we obtain: 

\noindent 
$(a)$  
Let $ u_j$  be an $ \L $-normalized Hecke--Maa{\ss} eigenform. 
Assuming the Generalized Lindel\"of Hypothesis we get for $ R 
\gg t_j^{-\delta+\epsilon} $ with $ \delta \leq 2/5 $: 
\begin{equation*}
\frac{1}{\vol(B_R (P))} \int_{B_R(P)} |u_j(Q)|^2 dv(Q) = 
\frac{1}{\vol(\GmodHthree)} + o_{P, \delta}(1).
\end{equation*}

\noindent 
(b) Let $ \theta $ be an exponent towards the Ramanujan--Petersson conjecture. 
For $ R \gg t^{-\delta+\epsilon} $ with $ \delta \leq \frac{1-2\theta}{34 + 
4\theta} $ we have: 
\begin{equation*}
\frac{1}{\log(1+t^2) \vol(B_R(P))} 
\int_{B_R(P)} \left|E\left(Q, 1+it\right)\right|^2 dv(Q) 
= 
\frac{|\mathcal{O}_K^{*}| \sqrt{|d_K|}}{4 \vol(\GmodHthree)} + o_{P, \delta}(1).
\end{equation*}
Assuming the Ramanujan conjecture $\theta=0$ we have equidistribution up to 
$ R \gg t^{-1/34+\epsilon} $. 

\noindent 
(c) For $  \Gamma = \pslzi $ we improve the exponent for the Eisenstein series
to $ \delta \leq \frac{1-2\theta}{27+2\theta}$.
\end{theorem}

\begin{remark} 
In both cases $(b)$ and $(c)$ of Theorem \ref{theorem2}, assuming the Generalized 
Lindel\"of Hypothesis we obtain QUE up to scale $ R \gg t^{-2/5+\epsilon}$. 
\end{remark}

Currently, the best known exponent $ \theta=7/64 $ is due to 
Nakasuji~\cite[Cor.~1.2]{nakasuji}. It seems possible that the growth 
of the sup-norm of Hecke--Maa\ss\ forms is the only obstacle causing 
QUE to fail for lower regimes and that $ \delta = 3/4 $ is the optimal 
exponent for QUE of Hecke--Maa{\ss} forms for all shrinking balls in 
the $ 3 $-dimensional case. Thus, we expect that for any center point $ P $ 
and for $ R \gg t_j^{-\delta} $ with  $ \delta < 3/4 $ 
we have quantum unique ergodicity:
\begin{equation*}
\frac{1}{\vol(B_R(P))} \int_{B_R(P)} |u_j(Q)|^2 dv(Q) = 
\frac{1}{\vol(\GmodHthree)} + o_{P, \delta}(1). 
\end{equation*}
Since Laplace eigenfunctions with large sup-norm are expected to form a thin subsequence of the 
discrete spectrum, the question of quantum ergodicity up to the Planck-scale remains open.

\subsection{Quantum variance estimates for shrinking balls of the Picard 
manifold}

Combining ideas and methods from the proofs of Theorems~\ref{theoremeisend2} 
and \ref{theorem2} we can further improve the exponent of the shrinking radius 
$R \gg t^{-\delta}$ for the Eisenstein series on the Picard manifold $ \pslzi 
\setminus \mathbb{H}^3 $. In this case, the slightly better exponent in 
Theorem~\ref{theorem2} follows from a large sieve of Watt~\cite{watt}, currently 
known only for congruence subgroups of the Picard group. Applying again this 
sieve and a mean Lindel\"of estimate for the second integral moment of the Hecke 
$L$-function $ L(s, u_j) $ we obtain a uniform upper bound for the quantum variance 
of Eisenstein series in shrinking balls of the Picard manifold. This can be 
considered as an analogue of Theorem~\ref{theoremeisend2} in dimension~x$3$. 

\begin{theorem} \label{theoremeisend3} 
Let $ \G = \pslzi $ and $ a \in [0, 1] $ be the parameter related to the twelfth 
moment of Riemann zeta function as in (\ref{12thmomwitha}). The quantum variance 
of the Eisenstein series satisfies the uniform upper bound: 
\begin{equation} \label{eisenvariancebound} 
\int_{T}^{2T}  \left|\frac{\int_{B_R(P)} |E(Q,1+it)|^2 dv(Q)}{\log(1+t^2) 
\vol(B_R(P))} - \frac{2}{\vol(\GmodHthree)}\right|^2 dt 
\ll_{P} 
\frac{T^{-1+\epsilon}}{R^{5+\epsilon}} +\frac{T^{\frac{3a-7}{9} + 
\epsilon}}{R^{\frac{29+3a}{9} + \epsilon}} + \frac{T}{(\log \log T)^2}.
\end{equation}
Thus quantum ergodicity holds for shrinking balls of radii $ R \gg t^{-\frac{2}{5}+ \epsilon} $. 
\end{theorem}

\noindent 
Note that, as in Theorem \ref{theoremeisend2}, the first term comes 
from the discrete spectrum, the second from the continuous spectrum 
and the last one from the degenerate contribution. 

\begin{remark} 
The unconditional result $ a = 1 $ follows from Heath-Brown's result  
for the twelfth moment of the Riemann zeta function (see \cite{heathbrown}). 
\end{remark}

\subsection{Failure of QUE in shrinking sets on manifolds of large dimension} 
\label{failurendimensions}
 
In Section~\ref{sectionlargedimensions} we generalize Theorem~\ref{proposition1} 
to higher dimensions. For specific discrete arithmetic groups $ \Gamma \subset 
SO(n,1) $ acting on the classical $ n $-dimensional hyperbolic space $ \mathbb{H}^n $, 
$ n \geq 4 $, and $ \mathcal{M} = \Gamma \setminus \mathbb{H}^n $ we prove that 
quantum unique ergodicity in shrinking balls centred at arithmetic points fails 
for radii
\begin{equation} \label{lowerboundfailque} 
R \asymp t_j^{-\delta_n}
\end{equation}
for some $ \delta_n < 1 $. For $ n \geq 5 $ we derive (\ref{lowerboundfailque}) 
with the explicit exponent $ \delta_n = \frac{5}{n+1} $. This follows from the $ \Omega $-result 
of Donnelly~\cite{donnelly} for the sup-norm of Laplace eigenfunction on 
$ n $-dimensional arithmetic hyperbolic manifolds. Thus we use again, though 
indirectly, specific theta lifts constructed in \cite{donnelly}. The case of 
$ \Omega $-results for the sup-norm problem in dimension $ n=4 $ is covered 
by the general result of Brumley and Marshall~\cite{brumley} with an unspecified 
exponent $\delta = \delta_4$. Since $ \delta_n \rightarrow 0 $ as $ n \rightarrow \infty $ we infer
that quantum unique ergodicity fails in balls shrinking rapidly in terms of the 
dimension. An analogous phenomenon was proved by Lester and Rudnick~\cite{lesterrudnick} 
for the $n$-dimensional Euclidean torus $ \mathbb{T}^n $, $ n\geq 4 $, who 
proved the existence of such {\lq massive irregularities\rq} on Euclidean circles 
of radii
\begin{equation*}
R \asymp t_j^{-\frac{1}{n-1}-\epsilon}. 
\end{equation*}
The method of Lester and Rudnick~\cite[Sect.~6]{lesterrudnick} is very arithmetic 
in nature, relying on lattice counting arguments and estimates for representations 
of positive definite binary quadratic forms. Our proof is more spectral in nature; 
one can argue that the arithmetic part of the proof is present in the constructions 
of theta lifts in \cite{brumley}, \cite{donnelly}, \cite{milicevic}, \cite{rudnicksarnak}.
For $n$-dimensional compact Riemannian manifolds Han \cite[Thm.~3]{han3} recently 
proved an analogous result for non-equidistribution at shrinking balls centred around 
points with large eigenfunction values.

\begin{remark}  
As we emphasized earlier the exponent $ \delta $ for the equidistribution of 
Hecke--Maa\ss\ forms on $3$-manifolds is better than the exponent obtained for 
the modular group. This can be roughly justified as follows: for dimension
$ n = 2, 3 $ the exponent follows (under GLH) from an estimate of the form 
\begin{equation} \label{hmformsexpcomparison}
\frac{T}{ R^{2n-1}} = o\left(T^{n}\right) 
\Longrightarrow 
R \gg_{\epsilon} T^{-\frac{n-1}{2n-1} + \epsilon}.
\end{equation}
The exact behaviour of the optimal exponent under the Generalized Lindel\"of 
Hypothesis for higher dimension $ n \geq 4 $ remains open. 
\end{remark} 

\begin{remark}
Quantum ergodicity and arithmetic quantum unique ergodicity have also been 
studied for other rank one cases: Lindenstrauss's results~\cite{lindenstrauss} 
cover also the case of Hecke--Maa\ss\ forms on $ (\mathbb{H}^2)^n $, while 
Truelsen~\cite{truelsen} studied the quantum ergodicity of Eisenstein series 
for $ (\mathbb{H}^2)^n $. It is an interesting question to investigate the 
quantum unique ergodicity problem on shrinking balls for these Hilbert modular 
surfaces. 
\end{remark}

\subsection{Acknowledgments} 
The authors would like to thank Gautami Bhowmik, Roelof Bruggeman, Dmitry Frolenkov, 
Niko Laaksonen, Djordje Mili\'cevi\'c, Yiannis Petridis, Gabriel Rivière, Brian Winn 
and in particular 
Peter Humphries for useful conversations on equidistribution results, quantum 
ergodicity on thin sets and moments of $L$-functions. The first author was 
supported by the Labex CEMPI (ANR-11-LABX-0007-01) and is currently supported 
by an IdEx postdoctoral fellowship at IMB, University of Bordeaux. 
The second author is currently supported by a ENS Lyon CDSN PhD Scholarship and the Labex CEMPI.
The third author was supported in part by the Labex CEMPI (ANR-11-LABX-0007-01).

\section{Proof of Theorem \ref{omegatheoremeisenstein}}

In this section we prove Theorem~\ref{omegatheoremeisenstein}. For this let 
$ z = -\tfrac{b}{2a} + \tfrac{i}{2a} \sqrt{|d|}$ be a fixed Heegner point 
for the modular surface $ \pslz \setminus \mathbb{H}^2 $ and denote by 
$ q(x, y) = ax^2 + bxy + cy^2 $ the positive definite binary quadratic form 
of discriminant $ d = b^2 - 4ac < 0 $ associated to the Heegner point $ z $. 
We now define the positive definite rational binary quadratic form $ Q $ by 
$ Q(m, n) := |mz + n|^2 $ (cf.\ \cite{milicevic2d} or \cite{zagier2}) and 
consider the Epstein zeta function $ Z(s, Q) $ associated to this quadratic 
form $ Q $ which is defined by 
\begin{align}
Z(s, Q) = \sum_{\substack{(m, n) \in \Z^2, \\ (m, n) \neq (0, 0)}} Q(m,n)^{-s} 
= \sum_{n} \frac{r_Q(n)}{n^s}, \quad \Re(s) > 1.
\end{align}
Here $ r_Q(n) $ denotes the number of representations of 
$ n $ by $ Q $. The Eisenstein series $ E(z, s) $ can be given with the 
help of this Epstein zeta function, namely we have $ \zeta(2s) E(z,s) = 
\Im^s(z) Z(s, Q) $ (cf.\ e.g.\ \cite{zagier2}). Selberg made the important 
discovery that an eigenfunction of the Laplace operator with eigenvalue 
$ \lambda = \tfrac{1}{4} + t^2 $, $ t \in \C $ is also an eigenfunction of 
every invariant integral operator and the corresponding eigenvalue depends 
only on the original eigenvalue $ \lambda $ and the kernel of the integral 
operator. If this kernel is given by $ k \circ \rho $ where $ \rho $ 
incorporates the hyperbolic distance, then the new eigenvalue is $ h(t) $.  
The Selberg transform $ h(t) $ can be calculated from the kernel using 
integral transformations (see e.g.\ \cite[Eq.~(1.60), (1.62)]{iwaniec}). 
Applying this to the characteristic kernel 
\begin{equation*}
k_R(u) = \frac{1}{\vol(B_R)}\cdot \chi_{[0, R]}(u)
\end{equation*} 
and keeping in mind that the Eisenstein series $ E(z, \tfrac{1}{2}+it) $ 
is an eigenfunction of the Laplace operator with eigenvalue $ \frac{1}{2} 
+it $, we obtain 
\begin{equation}\label{lowerboundeisenstein} 
\frac{1}{\vol(B_R(w))} \int_{B_R(w)} E(z, 1/2+it) d\mu(z) = 
h_R(t) \frac{\Im(w)^{1/2+it}}{ \zeta(1+2it)} Z(1/2+it, Q) 
\end{equation}
where $ h_R(t) $ is the Selberg transform of $ k_R $ defined by 
\cite[Eq.~(1.62)]{iwaniec}. The right-hand side of (\ref{lowerboundeisenstein}) 
can now be estimated as follows: first of all, we have $ h_R(t) = \Omega\left((Rt)^{-3/2}\right)$ 
as $ Rt \to \infty $ by \cite[Lemm.~4.2]{humphries}. Furthermore, 
\cite[Thm.~3]{fomenko} implies the $ \Omega $-result
\begin{align}
Z(1/2+it, Q) = 
\Omega\left(\exp\left(C' \sqrt{\frac{\log t}{\log \log t}}\right)\right)
\end{align}
for some constant $ C' > 0 $ depending on the quadratic form $ Q $, i.e.\ on 
the Heegner point $ z $. Using as well Vinogradov's bound
\begin{equation*}
\zeta(1+2it) \ll (\log t)^{2/3}
\end{equation*}
(see \cite[Cor.~8.28]{iwanieckowalski}) and the Cauchy--Schwarz inequality, we 
then infer 
\begin{equation*}
\begin{split} 
\frac{1}{\log\left(\frac{1}{4}+t^2\right) \vol(B_{R}(w))} \int_{B_{R}(w)} 
\left|E\left(z,\frac{1}{2}+it\right)\right|^2 d \mu(z) 
&\gg_w \frac{|h_R(t)|^2}{\log t} \left|\frac{Z(1/2+it, Q)}{\zeta(1+2it)}\right|^2 \\
& = \Omega_w \left( \frac{\left|Z(1/2+it, Q)\right|^2}{(R t)^3 (\log t)^{7/3}} \right).
\end{split} 
\end{equation*}
The statement now follows if we choose $C=2C'/3$ in (\ref{upperboundforf}).

\bigskip 

Conjecturally we have the stronger bound 
$ \zeta(1+2it) \ll \log \log t $ (for instance, this follows from the Riemann Hypothesis) 
which would allow us to improve the bound (\ref{upperboundforf}).

\section{Young's machinery for the modular surface and product formulae} 
\label{sectionthree} 

Before giving the proof of Theorem \ref{theoremeisend2} we describe Young's 
approach to the shrinking balls problem on the modular surface and recall some 
basic background on triple product formulae.

\subsection{Young's method for $ \GmodH $} 

Let $ \phi = \phi_R $ be a test function that satisfies for every $ k \geq 1 $ 
\begin{align} \label{testphiproperty} 
\|\Delta^k \phi\|_1 \ll_k R^{-2k}. 
\end{align}
We can consider $ \phi $ as a smooth approximation for the characteristic 
function of $ B_R $ and for $ R \geq 0 $ we pick a family $ (\phi_R)_R $ 
of such test functions with the property $ \phi_R \rightarrow \chi_{B_R}$ as 
$ R \rightarrow 0 $. Young's main result for the case of Hecke--Maa{\ss} 
forms case is summarized in the following proposition: 

\begin{proposition}[Young \cite{young}] \label{propyoungM}
Let $ u_j $ be a Hecke--Maa{\ss} cusp form on the modular surface with 
Laplace eigenvalue $ \tfrac{1}{4} + t_j^2 $. Furthermore, let $ \phi = 
\phi_R $ be a fixed test function as above with $ R \gg t_j^{-\delta} $ 
for some fixed $ 0 < \delta < 1 $. Assuming the generalized Lindel\"of 
hypothesis (GLH) we have, for any $ M \geq 1 $, 
\begin{align} \label{youngexplicit} 
\int_{\GmodH} \phi(z) |u_j(z)|^2 d\mu(z) = 
\int_{\GmodH} \phi(z) d\mu(z) 
+ O_{\epsilon} \left(\|\phi\|_2 R^{-1/2} t_j^{-1/2+\epsilon}\right) 
+ O_M\left(\|\phi\|_1 t_j^{-M}\right). 
\end{align}
\end{proposition}

Approximating $ \chi_{B_R} $ by $ \phi_R $, normalizing the appearing integrals 
and using the asymptotic $ \|\phi\|_2 \asymp R $, $ \|\phi\|_1 \asymp R^2 
\asymp \vol(B_R) $ then allows us to rewrite (\ref{youngexplicit}) as
\begin{equation} \label{youngexplicitballsversion} 
\frac{1}{\vol(B_R(w))}\int_{B_R(w)} |u_j(z)|^2 d \mu(z) = 
\frac{1}{\vol(\GmodH)} 
+ O_{\epsilon}\left(R^{-3/2} t_j^{-1/2+\epsilon}\right) 
+ O_{M}\left(t_j^{-M}\right).
\end{equation}
Part $ (a) $ of Theorem~\ref{younghumphries} now follows immediately from 
this identity. The proof of Proposition~\ref{propyoungM} requires spectral 
theory and triple product formulae.

\subsection{Triple product formulae for $\slz \setminus \mathbb{H}^2$ and 
regularization of Eisenstein series integrals} \label{watsonformuladimension2}

Let $ u_j $ and $ u $ be two Hecke--Maa{\ss} cusp forms for the modular group 
with corresponding eigenvalues $ \tfrac{1}{4}+t_j^2 $ and $ \tfrac{1}{4}+t^2 $. 
Based on previous works of Garret, Harris, Kudla, Piatetski-Shapiro, Rallis, to 
name only a few, Watson~\cite{watson} proved a formula relating 
\begin{equation*} 
\left<|u_j|^2, u \right> := \int_{\GmodH} u(z) |u_j(z) |^2 d\mu(z) 
\end{equation*} 
to a triple product of $ L $-functions associated to $ u_j $ and $ u $ and thus 
relating a priori the QUE~conjecture to subconvexity bounds for $ L $-functions. 
More precisely, his formula reads as follows: 
\begin{equation} \label{tripleproductcompletedfunctions}
\left|\left<|u_j|^2, u \right>\right|^2 =  
\frac{1}{8} \frac{\Lambda(1/2, \Sym^2 u_j \otimes u) 
\Lambda(1/2, u)}{\Lambda(1, \Sym^2 u) \Lambda(1, \Sym^2 u_j )^2} 
\end{equation} 
where $ \Lambda $ denotes the completed $ L $-functions. Replacing the completed 
$ L $-functions by their definition, we see that 
the right-hand side of (\ref{tripleproductcompletedfunctions}) can be written as 
a product of the non-Archimedean parts of the $ L $-functions with a product of 
Gamma factors. In the case $ u $ is an even form we get (\ref{tripleproductcompletedfunctions}) is equal to
\begin{equation} \label{watsonformula2}
\gamma_2(t_j,t) 
\frac{ L\left(1/2,\Sym^2 u_j \otimes u\right) L\left(1/2, u\right)}{L\left(1, 
\Sym^2  u\right) L\left(1,\Sym^2 u_j \right)^2} 
\end{equation}
where $ \gamma_2(t_j,t) $ satisfies
\begin{equation*} \label{gammafactorsgrowth-1} 
\gamma_2(t_j,t) 
\asymp  
\frac{\left|\Gamma\left(\frac{1}{4}+i\frac{t}{2}\right)\right|^4 
\left|\Gamma\left(\frac{1}{4}+i\left(t_j+\frac{t}{2}\right)\right)\right|^2 
\left|\Gamma\left(\frac{1}{4}+i\left(t_j-\frac{t}{2}\right)\right)\right|^2}
{\left|\Gamma\left(\frac{1}{2}+it_j\right)\right|^4 
\left|\Gamma\left(\frac{1}{2}+it\right)\right|^2}
\end{equation*}
for large $t_j$ and $t$. Using Stirling's formula we obtain the asymptotic 
\begin{equation*} \label{gammafactorsgrowth-2} 
\gamma_2(t_j,t) \asymp \frac{\exp(\frac{\pi}{2} 
\left(\mathcal{Q}(t_j,t)\right))}{\mathcal{P}_2(t_j,t)} 
\end{equation*}
as $t_j, t \to \infty$ where 
\begin{equation} \label{definitionq}
\mathcal{Q}(t_j,t) = 4 |t_j| - |2t_j+t| - |2t_j-t|=\left\{ \begin{aligned}
&0, & \mbox{if} &  \quad 2 t_j>t >0,\\
& 4 t_j- 2 t, & \mbox{if} &  \quad 2 t_j\leq  t,
\end{aligned} \right.
\end{equation}
and
\begin{equation} \label{definitionp1}
\mathcal{P}_2(t_j,t) = (1+|t|) ( 1+|2t_j+t|  )^{1/2} (1+|2t_j - t|  )^{1/2} 
\end{equation}
(see also \cite[Eq.~(4.2)]{young}). By Hoffstein-Lockhart \cite[Thm.~0.1]{hoffstein} the $L$-values 
$ L(1, \Sym^2 u) $ and $ L(1, \Sym^2 u_j) $ are of moderate growth, namely we 
have 
\begin{equation*}
t_j^{-\epsilon} \ll L(1, \Sym^2 u_j ) \ll t_j^{\epsilon}, \quad 
t^{-\epsilon} \ll L(1, \Sym^2 u ) \ll t^{\epsilon}.
\end{equation*}
As the convexity bound for $ L(1/2, u \otimes \Sym^2 u_j) $ is 
\begin{equation*}
L(1/2, u \otimes \Sym^2 u_j) 
\ll_{\epsilon} \left(t_j^{2}+t \right)^{\frac{1}{2}+\epsilon},  
\end{equation*}
we see that any subconvexity bound of the form $ o(t_j) $ implies the QUE 
conjecture. In particular, the generalized Riemann hypothesis (GRH) implies 
the QUE conjecture with the predicted rate of convergence. 
A similar product formula to (\ref{tripleproductcompletedfunctions}) holds 
if we replace the Hecke--Maa{\ss} cusp form by an Eisenstein series: 
\begin{align} \label{producteisensteinseriesmodulargroup}
\left|\left<\left| u_j \right|^2, E(\cdot, 1/2+it) \right>\right|^2 
= \frac{1}{4} \frac{|\Lambda(1/2+it)|^2 |\Lambda(1/2+it, \Sym^2 u_j)|^2}
{|\Lambda(1+2it)|^2 \Lambda(1, \Sym^2 u_j )^2}
\end{align}
(see \cite[Eq.~(17)]{luo_quantum_1995} or \cite[Prop.~2.8]{humphries}). 
The Archimedean part of the product appearing on the right-hand side of 
this identity is similar to the previous one appearing in 
(\ref{watsonformula2}) and has an asymptotic behaviour 
$ \asymp \gamma_2(t_j,t)$. 

\bigskip 

It is well-know that apart from the discrete part the spectrum of 
the Laplace operator on $ L^2(\GmodH) $ has also an absolutely continuous 
part given by the Eisenstein series. As in the case of the Hecke--Maa{\ss} 
forms we define now a measure involving the Eisenstein series as follows: 
\begin{equation} \label{dmutdef}
d\mu_t(z)  = |E(z, 1/2+it)|^2 d \mu(z).
\end{equation} 
The inner product of $ |E(z, 1/2+it)|^2 $ with the Hecke--Maa{\ss} cusp form 
$ u_j $ of Laplace eigenvalue $ \tfrac{1}{4} + t_j^2 $ is explicitly given 
by a product of $ L $-functions (see \cite[Eq.~(17)]{luo_quantum_1995}, 
\cite[Eq.~(4.3)]{young}) and we have the product formula 
\begin{equation} \label{gammafactorseisensteincont}
\begin{split} 
\left|\left<\left|E(\cdot, 1/2+it)\right|^2, u_j\right>\right|^2 
&= \left|\int_{\GmodH} u_j(z) d\mu_t(z)\right|^2 \\ 
&= \frac{1}{2} \frac{\Lambda(1/2, u_j)^2 |\Lambda(1/2 +2it, u_j)|^2}{|\Lambda(1+2it)|^4 
\Lambda(1, \Sym^2 u_j)}. 
\end{split} 
\end{equation}
The Gamma factors appearing in (\ref{gammafactorseisensteincont}) behave as 
$ \asymp \gamma_2(t,t_j) $ as $ t_j, t \to \infty $. However, if we replace 
the Hecke--Maa{\ss} cusp form $ u_j $ by an Eisenstein series in 
(\ref{gammafactorseisensteincont}), then the integral does not converge 
anymore.  
In order to overcome this technical difficulty we use Zagier's theory for 
Rankin--Selberg integrals for functions that are not of not rapid decay but 
satisfy mild growth conditions \cite{zagier}. This method was already used 
in \cite{young} and \cite{humphries} and consists basically of regularizing the 
appearing integrals appropriately. Let $ F $ be a $ \Gamma $-invariant 
function that satisfies the growth condition 
\begin{equation} \label{eq:F} 
F(z) = \varphi(y) + O(y^{-N})
\end{equation}
for any $ N > 0 $ as $ y \to \infty $ where 
\begin{equation*} 
\varphi(y) = \sum_{i=1}^{m} \frac{c_i}{n_i!} y^{a_i} \log^{n_i} y, 
\quad a_i, c_i \in \C, \ n_i \geq 0. 
\end{equation*} 
Then the regularized integral of $ F $ is defined as 
\begin{equation*} 
R.N. \int_{\GmodH} F(z) d\mu(z) 
:=  \int_{\GmodH} \left(F(z) - \mathcal{E}(z)\right)  d\mu(z) 
\end{equation*}
where $ \mathcal{E}(z) $ is a suitable linear combination of Eisenstein 
series and derivatives of Eisenstein series that can be explicitly given 
taking into account the $ a_i, c_i $ and $ n_i $, namely we have 
\begin{equation*} 
\mathcal{E}(z) = 
\sum_{\alpha_i \geq 1/2} c_i \frac{\partial^{n_i}}{\partial \alpha_i^{n_i}} 
E(z, \alpha_i) 
\end{equation*} 
(see \cite{zagier}, p.~427). If 
$ 
F(z) := E(z, 1/2 + it') \left|E(z, 1/2+it)\right|^2 
$ 
we obtain the regularized scalar product 
\begin{equation} \label{def:RN} 
\left<\left|E(\cdot, 1/2+it)\right|^2, E(\cdot,1/2+it')\right>_{\rn} := 
R.N. \int_{\GmodH} \left(F(z) - \mathcal{E}(z)\right)  d\mu(z). 
\end{equation} 
Zagier's results \cite[Eq.~(44)]{zagier} now give:
\begin{theorem} \cite{zagier} \label{lemmazagier}
The regularized triple product 
integral of Eisenstein series
\begin{equation} \label{regularintegrald2} 
\left<|E(\cdot, 1/2+it)|^2, E(\cdot,1/2+it') \right>_{\rn} 
\end{equation}
is equal to
\begin{equation}\label{eisensteinproduct}
 \, \frac{\Lambda(1/2-it')^2 \Lambda(1/2+i(2t-t')) 
\Lambda(1/2-i(2t+t'))}{|\Lambda(1+2it)|^2 \Lambda(1-2it')}. 
\end{equation}
\end{theorem}

Squaring and estimating the Gamma factors appearing in the functional 
equation of the Riemann zeta function we get the following asymptotic 
behaviour: 
\begin{equation*} 
\begin{split} 
& \left|\left< |E(\cdot, 1/2+it)|^2, E(\cdot,1/2+it') \right>_{\rn}\right|^2 \\
& \quad \quad 
\asymp \gamma_2(t,t') \frac{|\zeta(1/2-it')|^4 |\zeta(1/2+i(2t-t'))|^2 
|\zeta(1/2-i(2t+t'))|^2}{|\zeta(1+2it)|^4 |\zeta(1-2it')|^2}
\end{split} 
\end{equation*}
as $t, t' \to \infty$. We discuss analogous triple product formulae for 
arithmetic 3-manifolds in Subsections~\ref{watsonichinoformuladimension3} 
and \ref{regularizationintegrals3}.

\section{Quantum variance of Eisenstein series for shrinking balls \\ 
on the modular surface} \label{sectionquantumvareisen3}

For the rest of this section we denote by $\B = \B_{\G}$ a family of non-constant 
Hecke--Maa\ss\ cusp forms $ u_j \in L^2(\GmodH) $. The following theorem is an 
extension of the Plancherel formula for functions of moderate growth where an extra 
degenerate contribution appears naturally.

\begin{theorem}\cite[Eq.~(4.20), p.~243]{michelvenkatesh} \label{michelvenkateshprop}
If $ F $ is a smooth function on the modular surface of the type (\ref{eq:F}) 
with $ \Re(a_i) \neq 1/2 $, $ u_0 = \sqrt{3/\pi} $ is the $ L^2 $-normalized constant 
eigenfunction and $ G $ is smooth and compactly supported, then 
\begin{equation} \label{promichelvenk}
\begin{split} 
\left< F, G \right> 
&= \left<F, u_0^2\right>_{\rn} \left<1, G\right> 
+ \sum_{u_j \in \B} \left<F, u_j\right> \left<u_j, G\right> \\ 
& \quad 
+ \frac{1}{4\pi} \int_{-\infty}^{\infty} \left<F, E(\cdot, 1/2+it')\right>_{\rn} 
\left<E(\cdot, 1/2+it'), G\right> dt' + \left<\mathcal{E}, G\right>. 
\end{split} 
\end{equation}
\end{theorem}

\noindent 
Thus in order to bound 
\begin{equation} \label{varianceeisenstein2d}
\frac{1}{\vol(B_R(w))} \int_{B_R(w)} |E(z,1/2+it)|^2 d\mu (z) 
- \frac{\log(1/4+t^2)}{\vol(\GmodH)} 
\end{equation}
and prove Theorem~\ref{theoremeisend2} 
we use Theorem~\ref{michelvenkateshprop} with $ G = \phi_R $ and $ F(z) 
= E(z, s_1) E(z,s_2) $, $ s_1 = \overline{s_2} = \tfrac{1}{2} + it $. 
This approach using the spectral decomposition allows us to bound the various 
terms using the product formulae and the properties of $ \phi_R $.

\subsection{The contributions of the constant eigenfunction and the degenerate 
contribution}

The contribution of $ \left<\mathcal{E}, \phi_R \right> $ is typically 
one of the most delicate parts of the generalized Plancherel formula 
(\ref{promichelvenk}) and is related to the constant coefficient of 
the Eisenstein series. It follows from \cite[pp.~976, 980]{young} that 
\begin{equation*} 
\begin{split} 
\left<|E(\cdot, 1/2 + it)|^2, u_0^2\right>_{\rn} \left<1, \phi_R\right> 
& + 
\left< \mathcal{E}, \phi_R \right> 
= \\ 
& 
\log(1/4+t^2) \langle\phi_R, u_0^2\rangle 
+ 
O\left(\frac{\log t}{\log \log t} \|\phi_R\|_1\right) 
+ O\left(t^{-M}\right). 
\end{split} 
\end{equation*}
Hence the contribution of $ \left<\mathcal{E}, \phi_R\right> $ and the 
constant eigenfunction in (\ref{int116}) yield the leading term and an 
error term of the size of $ O\left(\frac{T}{(\log \log T)^2}\right) $.

\subsection{The contribution of the discrete spectrum} 
\label{contributionofthediscretespectrum}

If we apply Theorem~\ref{michelvenkateshprop} and use the properties of 
$ \phi_R $ (see (\ref{testphiproperty})) and (\ref{gammafactorseisensteincont}) 
to estimate the contribution of those cusp forms with large eigenvalues, 
we easily see that the contribution of the discrete spectrum, i.e.\ the 
contribution corresponding to $ u_j \in \B $ in the spectral expansion of 
(\ref{varianceeisenstein2d}), can be estimated as follows: 
\begin{equation*} 
\begin{split} 
\sum_{u_j \in \B} & \left<\left|E(z, 1/2+it)\right|^2, u_j\right> 
\left<u_j, \phi_R\right> \\ 
& 
\ll 
\|\phi_R\|_2 \left(\sum_{t_j \leq R^{-1} t^{\epsilon}}  
\left|\left<\left|E(z, 1/2+it)\right|^2, u_j\right>\right|^2\right)^{1/2} 
+ O_M\left(t^{-M}\right) \\
& 
\ll 
\|\phi_R\|_2 \left(\sum_{t_j \leq R^{-1} t^{\epsilon}} 
\frac{(1 + |t|)^{\epsilon} \left|L(1/2,u_j)\right|^2 
\left|L(1/2+2it,u_j)\right|^2}{(1+|t_j|)^{1 - \epsilon}
(1 + |t_j - 2t|)^{1/2} (1 + |t_j + 2t|)^{1/2}}\right)^{1/2} 
+ O_M\left(t^{-M}\right) 
\end{split} 
\end{equation*}
(cf.\ \cite{young}, p.~978, (4.23) and (4.24)). 
We therefore can bound the contribution of the discrete spectrum 
to the right-hand side of (\ref{int116}) by
\begin{equation*}
\begin{split} 
\int_{T}^{2T} \frac{1}{R^2} \sum_{t_j \leq R^{-1} t^{\epsilon}} 
& 
\frac{(1 + |t|)^{\epsilon} \left|L(1/2,u_j)\right|^2 
\left|L(1/2+2it,u_j)\right|^2}{(1+|t_j|)^{1 - \epsilon}
(1 + |t_j - 2t|)^{1/2} (1 + |t_j + 2t|)^{1/2}} \, dt \\
& \ll 
\frac{T^{\epsilon}}{R^2} \sum_{t_j \leq R^{-1} (2T)^{\epsilon}} 
\frac{\left|L(1/2,u_j) \right|^2}{(1+|t_j|)^{1 - \epsilon}}
\int_{T}^{2T} \frac{\left|L(1/2+2it,u_j) \right|^2}{(1 + 
|t_j - 2t|)^{1/2} (1 + |t_j + 2t|)^{1/2}} \, dt.
\end{split} 
\end{equation*}
Since $ t_j \leq R^{-1} (2T)^{\epsilon} = o(T) $ we get $ (1 + 
|t_j - 2t|)^{1/2} (1 + |t_j + 2t|)^{1/2} \asymp (1+|t|) $. 
Thus the sum is estimated as follows: following Huang~\cite[Section~3]{huang} 
we infer
\begin{equation*}
\int_{T}^{2T} \frac{\left|L(1/2+2it,u_j)\right|^2}{(1+|t|)} \, dt 
\ll T^{\epsilon}(1+|t_j|)^{\epsilon}.
\end{equation*}
Thus the mean-subconvexity estimate (see \cite[(4.25)]{young})
\begin{equation*}
\sum_{t_j \leq R^{-1} T^{\epsilon} } \left|L(1/2,u_j) \right|^2 
\ll 
R^{-2} T^{\epsilon}
\end{equation*}
and summation by parts imply 
\begin{equation*}
\sum_{t_j \leq R^{-1} T^{\epsilon}} 
\frac{\left|L(1/2,u_j)\right|^2}{(1+|t_j|)^{1 - \epsilon}} 
\ll R^{-1-\epsilon} T^{\epsilon}.
\end{equation*}
We finally obtain that the contribution of the discrete spectrum to the 
right-hand side of (\ref{int116}) is bounded by $ R^{-3 - \epsilon} T^{\epsilon} $. 

\subsection{The contribution of the continuous spectrum} 
\label{contributionofthecontinuousspectrum}

For estimating the contribution of the continuous spectrum, we apply the 
regularized integral formula from Theorem~\ref{lemmazagier}. Working as 
in \cite{young} we bound the contribution of the continuous spectrum by 
\begin{equation*} 
\begin{split} 
\int_{T}^{2T} & \left|\frac{1}{R^2} \int_{-\infty}^{\infty} 
\left<|E(z, 1/2+it)|^2 , E(z, 1/2+it')\right>_{\rn} \left< E(z, 1/2+it'), 
\phi_R \right> dt'\right|^2 dt \\
& \ll 
\frac{T^{\epsilon}}{R^{2 + \epsilon}} 
\int_{T}^{2T} \int_{|t'| \leq R^{-1} T^{\epsilon}} \frac{|\zeta(1/2-it')|^4 
|\zeta(1/2+it'+2it)|^2 |\zeta(1/2+it'-2it)|^2}{(1 + |t'|) 
(1+|2t+t'|)^{1/2} (1+|2t - t'|)^{1/2}} \, dt' dt 
\end{split} 
\end{equation*}
together with a small error term $ O\left(R^{-1} T^{-M}\right) $ (see
\cite[Eq.~(4.30)]{young}). In order to bound the last integral we use 
Ingham's bound for the fourth moment of the Riemann zeta function 
\begin{equation} \label{ingham}
\int_{T}^{2T}  |\zeta(1/2+it)|^{4}  dt \ll T^{1+\epsilon}. 
\end{equation} 
This bound and an application of the Cauchy--Schwarz inequality imply 
\begin{equation} \label{promichelvenkexpanded2-2}
\begin{split} 
\frac{T^{\epsilon}}{R^{2+\epsilon}} \int_{|t'| \leq R^{-1} T^{\epsilon}} 
& 
\frac{|\zeta(1/2-it')|^4 }{(1+|t'|)} 
\int_{T}^{2T} \frac{|\zeta(1/2+it'+2it)|^2 |\zeta(1/2+it'-2it)|^2}{(1 
+ |2t+t'|)^{1/2} (1+|2t - t'|)^{1/2}} \, dt \, dt' \\ 
& \ll  
\frac{T^{\epsilon}}{R^{2+\epsilon}} \int_{|t'| \leq R^{-1} T^{\epsilon}} 
\frac{|\zeta(1/2-it')|^4}{(1+|t'|)} 
\prod_{\pm} 
\left(\int_{2T \pm t'}^{4T \pm t'} \frac{|\zeta(1/2+it)|^4 }{1+t}  dt\right)^{1/2} 
dt' \\
& \ll 
T^{\epsilon} R^{-2-\epsilon} \int_{|t'| \leq R^{-1} T^{\epsilon}} 
\frac{|\zeta(1/2-it')|^4 }{1+|t'|} \, dt' \\ 
& \ll 
T^{\epsilon} R^{-2-\epsilon}. 
\end{split} 
\end{equation}

\subsection{Proof of Theorem~\ref{theoremeisend2}} 

Combining the bounds for the various contributions of the spectrum 
we can now prove Theorem~\ref{theoremeisend2}. Namely, applying 
Theorem~\ref{michelvenkateshprop} we get 
\begin{equation*}
\int_{T}^{2T} 
\left|\frac{\int_{B_R(w)} |E(z,1/2+it)|^2 d\mu (z)}{\log(1/4+t^2) 
\vol(B_R(w))} - \frac{1}{\vol(\GmodH)} \right|^2 dt 
\ll_{w}
\frac{T^{\epsilon}}{R^{3+\epsilon}} + \frac{T^{\epsilon}}{R^{2+\epsilon}} 
+ \frac{T}{(\log \log T)^2}.
\end{equation*}
The expression on the right-hand side of the inequality is bounded 
by $ o_{w}(T) $ if and only if 
\begin{equation*}
R \gg T^{-1/3 +\epsilon} \asymp t^{-1/3 +\epsilon}. 
\end{equation*}
Consequently, we have quantum ergodicity up to this scale.

\section{Estimates for the Selberg transform and failure of QUE 
away from the Planck-scale}

In this section we summarize some facts for the Selberg transform 
for $ \Gamma \setminus \mathbb{H}^3 $ and give the proof of 
Theorem~\ref{proposition1}.

\subsection{The Selberg transform} 

For $ P = z + rj $ and $ Q = z' + r'j \in \mathbb{H}^3 $ we set 
\begin{equation*}
\delta(P, Q) = \frac{|z - z'|^2 + r^2 + r'^2}{2rr'}. 
\end{equation*}
Then the hyperbolic distance $ \rho(P, Q) $ of $ P $ and $ P' $ is 
given by 
\begin{equation*}
\cosh \rho(P,Q)= \delta(P,Q)
\end{equation*}
(see \cite[Prop.~1.6]{elstrodt}). Furthermore, we define a point-pair 
invariant $ K(P, Q) = k \circ \delta(P, Q) $ by 
\begin{equation} \label{kkernel}
k\left(\delta(P, Q)\right) = 
\frac{1}{\vol(B_R)} \cdot \chi_{[0, R]}\left(\rho(P,Q)\right). 
\end{equation}
In the situation of the hyperbolic $3$-space the Selberg transform 
(see \cite[Ch.~3.5]{elstrodt}) of $ k $ is given by
\begin{align} \label{selbergharishchandra1}
h_{R}(t) = 
\frac{a}{\vol(B_R)} \int_{-R}^{R} (\cosh R - \cosh u) e^{itu} du 
\end{align}
where $ a $ is a constant (see \cite[Eq.~(4.2)]{phirud}). Note that 
$ \chi_{[0, R]} \circ \rho $ denotes the characteristic kernel of the 
distance $ \rho = \rho(P, Q) $, being $ 1 $ if the distance of $ P $ 
and $ Q $ is less than $ R $ and $ 0 $ otherwise. 

\begin{lemma} \label{hrtcalculation3dim}
There exists a constant $ c $ such that, as $ R \rightarrow 0 $ and 
$ Rt \rightarrow \infty $, we have
\begin{equation} \label{estimatehrt}
h_{R}(t) = \Omega \left( \left(Rt \right)^{2}\right). 
\end{equation}
\end{lemma}

We give a proof of Lemma~\ref{hrtcalculation3dim} in 
Section ~\ref{sectionlargedimensions} for the general $ n $-space with 
$ n \geq 2$.
The estimate (\ref{estimatehrt}) is useful due to 
the fact that an eigenfunction of the Laplace operator is also an 
eigenfunction of the invariant integral operator given by the point-pair 
invariant (\ref{kkernel}) \cite[Ch.~3.5, Thm.~5.3]{elstrodt}.

\begin{proposition}\label{meanvalueprop}
Let $u_j$ be a Hecke--Maa\ss\ form with eigenvalue $1+t_j^2$. Then
\begin{equation*}
\frac{1}{\vol(B_R(P))} \int_{B_R(P)} u_j(Q) dv(Q) = h_R(t_j) u_j(P).
\end{equation*}
\end{proposition}

\noindent 
We can now give the proof of Theorem~\ref{proposition1}.

\begin{proof}[Proof of Theorem~\ref{proposition1}] 
(a) Using the Cauchy--Schwarz inequality and Proposition~\ref{meanvalueprop} 
we get the lower bound
\begin{equation*}
\frac{1}{\vol(B_R(P))} \int_{B_R(P)} \left|u_j(Q)\right|^2 dv(Q) 
\gg |h_R(t_j)|^2 \left|u_j(P)\right|^2.
\end{equation*}
Now assume that $ R \ll t_j^{-\delta}$ for some $ \delta > 0 $. 
By Lemma~\ref{hrtcalculation3dim} we have
\begin{equation*}
|h_R(t_j)|^2 \left|u_j(P)\right|^2 
= \Omega\left(  \frac{\left|u_j(P)\right|^2}{(Rt_j)^{4}}\right) 
= \Omega\left(  \frac{\left|u_j(P)\right|^2}{t_j^{4-4\delta}}\right).
\end{equation*}
If 
$ \mathcal{M} = \Gamma \setminus \mathbb{H}^3 $ is of Maclachlan--Reid 
type and $ P $ is a fixed QCM-point on $ \GmodHthree $, then by the 
lower bound (\ref{milicevicbound}) we obtain 
\begin{align*}
\left|u_j(P)\right| =  
\Omega_{\epsilon}\left(t_j^{ \frac{1}{2}-\epsilon}\right). 
\end{align*}
We conclude that
\begin{equation*}
\frac{1}{\vol(B_R(P))} \int_{B_R(P)} \left|u_j(Q)\right|^2 dv(Q)
\end{equation*}
is unbounded for $\delta >3/4$. This completes the proof 
of part (a) of Theorem~\ref{proposition1}. 
\\
(b) Similarly to the lower bound (\ref{milicevicbound}),
Rudnick and Sarnak~\cite[Thm.~1.2]{rudnicksarnak} proved for specific 
compact manifolds $\Gamma \backslash \mathbb{H}^3$ and points $P$ the bound 
\begin{equation*}
\left|u_j(P)\right| = \Omega \left(t_j^{1/2}\right). 
\end{equation*}
The proof of part (b) follows. 
\end{proof}

\section{Product formulae on hyperbolic 3-manifolds}

\subsection{The Watson--Ichino triple formula formula for $ \slc $ and 
other product formulae} \label{watsonichinoformuladimension3}

Ichino \cite{ichino} proved a far reaching generalization of Watson's 
formula for higher rank reductive groups which was worked out explicitly 
by Marshall \cite{marshall} for the case of automorphic representations 
associated to $\slc$. In our case it simplifies to the following statement: 

\begin{theorem} \label{theoremtripleproductichinomarshall}
Let $ u_j, u $ be two Hecke--Maa{\ss} cusp forms for the Bianchi group 
$ \Gamma = \textup{PSL}_2(\mathcal{O}_K) $, $ \mathcal{O}_K $ being the 
ring of integers of an imaginary quadratic number field of class number 
one. Then there exists a constant $ C_{\Gamma} $ such that 
\begin{equation} \label{tripleproductcompletedfunctions3dimensions}
\left|\left<|u_j|^2, u \right>\right|^2 
= 
\left|\int_{\GmodHthree} u(P) dv_j(P)\right|^2 
= C_{\Gamma} \, \frac{\Lambda(1/2, u \otimes \Sym^2 u_j) 
\Lambda(1/2, u)}{\Lambda(1, \Sym^2 u) \Lambda(1, \Sym^2 u_j )^2}.
\end{equation} 
\end{theorem}

Note that by abuse of notation we denote by $ \Lambda(s, f) $, $ f 
= u, \Sym^2 u, \Sym^2 u_j $ or \mbox{$ u \otimes \Sym^2 u_j $}, respectively, 
the completed associated $ L $-functions, as in Section~\ref{sectionthree}. 
These of course are not identical to 
functions appearing in (\ref{tripleproductcompletedfunctions}). 
%
In order to determine the dependence of (\ref{theoremtripleproductichinomarshall}) 
on $ u $ and $ u_j $  we replace the completed $ L $-functions. Then 
the right-hand side of (\ref{tripleproductcompletedfunctions3dimensions}) 
is equal to
\begin{equation*}
C_{\Gamma} \, \gamma_3(t_j,t) \, \frac{L\left(1/2, u \otimes \Sym^2 u_j \right) 
L\left(1/2, u \right)}{L\left(1, \Sym^2 u\right) L\left(1,\Sym^2 u_j \right)^2} 
\end{equation*}
where the factor $ \gamma_3(t_j,t) $ is asymptotic to 
\begin{equation*}
\frac{\left|\Gamma\left(\frac{1+i t}{2}\right)\right|^4 
\left|\Gamma\left(\frac{1}{2}+i\left(t_j+\frac{t}{2}\right)\right)\right|^2 
\left|\Gamma\left(\frac{1}{2}+i\left(t_j-\frac{t}{2}\right)\right)\right|^2}
{\left|\Gamma\left(1+it_j\right)\right|^4 
\left|\Gamma\left(1+it\right)\right|^2}.
\end{equation*}
Note that by using similar arguments to the ones given in Hoffstein--Lockhart 
\cite{hoffstein} one sees that the $ L $-values $ L(1, \Sym^2 u_j ) $ are of 
moderate growth. Namely, we have  
\begin{equation*}
t_j^{-\epsilon} \ll L(1, \Sym^2 u_j ) \ll t_j^{\epsilon} 
\end{equation*}
(see for instance \cite[Cor.~7]{li}). The factor $ \gamma_3(t_j, t) $ is 
estimated using Stirling's formula and we see that 
\begin{equation} \label{gammafactorsgrowth2} 
\gamma_3(t_j, t) \asymp 
\frac{\exp(\frac{\pi}{2} \left(\mathcal{Q}(t_j,t)\right))}{\mathcal{P}_3(t_j,t)} 
\end{equation}
where $\mathcal{Q}(t_j,t)$ is defined in (\ref{definitionq}) and
\begin{equation} \label{definitionp2}
\mathcal{P}_3(t_j,t) = (1+|t|) (1+|t_j|)^{2}.
\end{equation}
Notice here the difference between $ \mathcal{P}_2 $ defined in 
(\ref{definitionp1}) and $ \mathcal{P}_3 $ which allows us to 
obtain a better exponent $ \delta $ for the equidistribution of 
Hecke--Maa{\ss} forms in $ 3 $-dimensional shrinking balls. 

\bigskip 

For reasons of completeness, we also write down explicitly the 
product formulae for automorphic forms in $ \L(\GmodHthree) $ 
for the Eisenstein series. We have 
\begin{equation} \label{productformula3Dformula2}
\left|\left<\left|u_j \right|^2, E(\cdot, 1+it) \right>\right|^2 
= 
\frac{|\mathcal{O}^*|^2}{2^4} \, \frac{|\Lambda_K(\tfrac{1+it}{2})|^2 
|\Lambda(\tfrac{1+it}{2}, \Sym^2 u_j)|^2}{|\Lambda_K(1 + it)|^2 
\Lambda(1, \Sym^2 u_j)^2}
\end{equation}
where $ \Lambda_K(s) $ denotes the completed Dedekind zeta function 
of $ K $ (see e.g.\ \cite[p.~479]{koyama1} or \cite{heitkamp}, \S~17). 
The Gamma factors of (\ref{productformula3Dformula2}) are identical 
to those appearing in (\ref{tripleproductcompletedfunctions3dimensions}) 
and have therefore the same asymptotic behaviour as $ \gamma_3(t_j,t) $ 
(see (\ref{gammafactorsgrowth2})). From \cite[p.~481]{koyama1} or 
\cite[sub.~3.1]{laaksonen} we also get 
\begin{equation} \label{tripleproducteisensteinseries3manifolds}
\left|\left<\left|E(\cdot, 1+it)\right|^2, u_j\right>\right|^2 
= 
\frac{|\mathcal{O}^*|^4 \sqrt{|d_K|}}{2^{6}} \, 
\frac{\Lambda(1/2, u_j )^2 |\Lambda(1/2 + it, u_j)|^2}
{|\Lambda(1 + it)|^4 \Lambda(1, \Sym^2 u_j)} 
\end{equation}
and the Archimedean part of (\ref{tripleproducteisensteinseries3manifolds}) 
grows like $ \gamma_3(t,t_j) $. 

\subsection{Regularization of Eisenstein integrals for 3-manifolds} 
\label{regularizationintegrals3}

In order to handle the problem of the divergence of the integral 
\begin{equation} \label{divergingintegral}
\int_{\Gamma \setminus \mathbb{H}^3} 
E(P, 1+it') \left|E\left(P, 1+it\right)\right|^2 dv(P) 
\end{equation}
we use Zagier's regularization as in Subsection~\ref{watsonformuladimension2}. 
In the 3-dimensional case the renormalized scalar product 
$ \langle \cdot, \cdot\rangle_{\rn} $ is defined analogously to the 
2-dimensional case by simply subtracting the divergence causing 
part in the Fourier expansion. Then using the same arguments as in 
\cite[pp.~429--430]{zagier} as well as the calculations of 
\cite[p.~8]{laaksonen} we obtain:

\begin{lemma} \label{eisensteinproduct3dim}
Let $ \G = \mathrm{PSL}_2(\mathcal{O}_K) $ be a class number one 
Kleinian group and $ E(P,s) $ the Eisenstein series corresponding 
to the cusp $ \infty $. Then 
\begin{equation} \label{regularintegralkleinian} 
\begin{split} 
& \left<\left|E\left(\cdot, 1+it\right)\right|^2, E(\cdot,1+it')\right>_{\rn} 
= \\ 
& \phantom{E(\cdot, 1+it), E(\cdot,1+it')} 
\frac{|\mathcal{O}^*|^3 \, \sqrt{|d_K|}}{2^4} \, 
\frac{\Lambda_K^2\left(\frac{1+it'}{2}\right) \Lambda_K\left(\frac{1+it'}{2}-it\right) 
\Lambda_K\left(\frac{1+it'}{2}+it\right)}{\Lambda_K\left(1+it'\right) |\Lambda_K(1+it)|^2}. 
\end{split} 
\end{equation}
\end{lemma}

We mention that Zagier's method in the general 
GL$_2 $ case has been worked out by Wu \cite[Thm.~3.5]{wu2}. 
Our case can be worked out directly without recourse to Wu's work. 

\bigskip 

\noindent 
From (\ref{regularintegralkleinian}) we conclude
\begin{equation} \label{regularintegralkleinian-2} 
\begin{split} 
& 
\left|\left<\left|E\left(\cdot, 1+it\right)\right|^2, 
E(\cdot,1+it')\right>\right|^2 \\ 
& 
\hspace{5 ex} 
\asymp 
\gamma_3(t,t') \frac{\left|\zeta_K((1+it')/2)\right|^4 
\left|\zeta_K(1/2+i(t+t'/2))\right|^2 
\left|\zeta_K(1/2+i(t-t'/2))\right|^2}{|\zeta_K(1+it)|^4 |\zeta_K(1+it')|^2}.
\end{split} 
\end{equation}

\section{Equidistribution of Hecke--Maa{\ss}  forms on $\GmodHthree$}

In this section we give the proof of Part~$(a)$ of Theorem~\ref{theorem2}. 
The key tool of the proof is the triple product formula of Ichino. Since we 
follow the steps of the proof for the case of the modular surface we omit 
standard technicalities trying instead to concentrate on the new ingredients 
coming into the play. 

\bigskip 

We start by adjusting Young's machinery for the shrinking balls problem in 
the hyperbolic 3-space. Let $ \phi = \phi_R $ be a smooth approximation for 
the characteristic function of $ B_R $, i.e.\ we take a test function that 
for every $ k \geq 1 $ satisfies
\begin{equation} \label{testphiproperty3} 
\|\Delta^k \phi\|_1 \ll_k R^{-2k}.
\end{equation}

As before, we denote by $ \B = \B_{\Gamma} $ a family of non-constant Maa{\ss} 
cusp forms for the Kleinian group $ \G $. We obtain the following estimate: 

\begin{proposition}\label{propyoungMaass3D}
Let $ \Gamma = \textup{PSL}_2(\mathcal{O}_K) $ be a Bianchi group and assume 
that the class number of $ K $ is one. Furthermore, let $ \phi = \phi_R $ 
be a fixed test function as in (\ref{testphiproperty3}) and $ u_j $ a 
Hecke--Maa{\ss} cusp form on the arithmetic 3-manifold $ \Gamma \setminus 
\mathbb{H}^3 $. Then, assuming $ R \gg t_j^{-\delta} $ for some fixed 
$ 0 < \delta < 1 $ as well as GLH, we have for any $ M \gg 1 $
\begin{equation} \label{propyoung3Dexplicit}
\int_{\GmodHthree} \phi_R(P) |u_j(P)|^2 dv(P) 
= 
\int_{\GmodHthree} \ \phi_R(P) dv(P) 
+ O_{\epsilon}\left(\|\phi \|_2 R^{-1-\epsilon} t_j^{-1+\epsilon}\right) 
+ O\left(\|\phi\|_1 t_j^{-M} \right). 
\end{equation}
\end{proposition}

\begin{proof} 
Let $ u_j \in\B $. In the class number one case the spectral theorem 
(\cite[Sect.~6.3, Thm.~3.4]{elstrodt}) implies 
\begin{equation*} 
\begin{split} 
|u_j|^2 &= 
\frac{\left<|u_j|^2,1\right>}{\vol(\Gamma\setminus \H^3)} 
+ \sum_{u_m\in\B} \left<|u_j|^2,u_m\right>u_m \\ 
& \qquad 
+ \frac{1}{\pi \sqrt{|d_K|} |\mathcal{O}_K^*|}
\int_{-\infty}^{+\infty}\left<|u_j|^2,E\left(\cdot,1+it\right)\right>E\left(\cdot,1+it\right)dt.
\end{split} 
\end{equation*}
Since $ \f_R $ is a smooth compactly supported function on $ \H^3 $ we get
\begin{equation*} 
\begin{split} 
\langle \phi_R , |u_j|^2 \rangle 
&= \frac{\left<\f_R, 1 \right>\left<|u_j|^2,1\right>}{\vol(\Gamma\setminus \H^3)} 
+ \sum_{u_m\in\B}\left<|u_j|^2,u_m\right> \langle \phi_R , u_m  \rangle \\
&\quad 
+ \frac{1}{\pi \sqrt{|d_K|} |\mathcal{O}_K^*|} \int_{-\infty}^{+\infty}\left<|u_j|^2, 
E\left(\cdot,1+it\right)\right>\langle \phi_R ,  E\left(\cdot,1+it\right) \rangle dt. 
\end{split} 
\end{equation*}
In order to estimate $ \langle |u_j|^2, u_m\rangle $ we use the Watson--Ichino formula 
(see Theorem~\ref{theoremtripleproductichinomarshall}). 
Using Stirling's formula as well as the standard bound $ \|u_m\|_\infty \ll 
(1+ t_m^2)^{1/2} $ for the sup-norm of $ u_m $ we cut the sum and the range 
of integration to conclude 
\begin{equation} \label{fourierexp} 
\begin{split} 
\langle\phi_R , |u_j|^2\rangle - \frac{ \left<\f_R, 1 \right>}{\vol(\Gamma\setminus \H^3)} 
& \ll 
\sum_{\substack{u_m \in \B, \\ t_m \leq t_j + c\log t_j}} \left<|u_j|^2,u_m\right> 
\langle \phi_R , u_m  \rangle \\
& \quad 
+ \frac{1}{\pi \sqrt{|d_K|} |\mathcal{O}_K^*|} \int_{|t|\leq t_j + C \log t_j} 
\left<|u_j|^2,E\left(\cdot,1+it\right)\right> 
\langle \phi_R ,  E\left(\cdot,1+it\right) \rangle dt \\ 
& \quad 
+ O_M\left(\|\phi\|_1 t_j^{-M}\right) 
\end{split} 
\end{equation}
for sufficiently large constants $c, C \gg 1 $ and $M \gg 1$. The initial assumption 
(\ref{testphiproperty3}) for the test function $\f_R$ allows us to restrict the sum 
and the range of integration in (\ref{fourierexp}) further as (\ref{testphiproperty3}) 
implies
\begin{equation*}
\left<\f, u_m\right> 
\ll \frac{\|\Delta^k \f\|_1 \|u_m\|_\infty}{(1+t_m^2)^{k}}  
\ll \frac{R^{-2k}}{(1+t_m^2)^{k-1/2}}.
\end{equation*}
We assume $ R \gg t_j^{-\delta}$ hence, up to the cost of a small error term, we can 
cut the sum and the integration in (\ref{fourierexp}) at $ t_m \leq  R^{-1} t_j^\epsilon $ 
so that 
\begin{equation*} 
\begin{split} 
\langle\phi_R , |u_j|^2\rangle - \frac{ \left<\f_R, 1 \right>}{\vol(\Gamma\setminus \H^3)} 
& \ll 
\sum_{\substack{u_m \in \B, \\ t_m \leq R^{-1} t_j^{\epsilon}}} 
\left<|u_j|^2,u_m\right> \langle \phi_R , u_m  \rangle \\ 
& 
+ \quad 
\frac{1}{\pi \sqrt{|d_K} |\mathcal{O}_K^*|} \int_{|t| \leq R^{-1} t_j^{\epsilon}} 
\left<|u_j|^2,E\left(\cdot,1+it\right)\right> 
\langle \phi_R, E\left(\cdot,1+it\right) \rangle dt \\ 
& 
+ \quad 
+ O\left(\|\phi\|_1 t_j^{-M}\right). 
\end{split} 
\end{equation*}
In order to bound the discrete contribution note that the Cauchy--Schwarz inequality 
implies 
\begin{equation*}
\sum_{\substack{u_m \in \B, \\ t_m \leq R^{-1} t_j^\epsilon}} 
\left<|u_j|^2, u_m\right>\langle \phi_R , u_m  \rangle 
\ll \left(\sum_{\substack{u_m \in \B, \\ t_m \leq R^{-1} t_j^\epsilon}} 
\left|\left<|u_j|^2, u_m\right>\right|^2\right)^\frac{1}{2} 
\left(\sum_{\substack{u_m \in \B, \\ t_m\leq R^{-1} t_j^\epsilon}} 
|\left<\f_R,u_m\right>|^2\right)^\frac{1}{2}.
\end{equation*}
The second term is bounded by $ \|\f_R\|_2 $. The first one is estimated 
using (\ref{tripleproductcompletedfunctions3dimensions}). Using Stirling's 
formula and assuming GLH we have 
\begin{equation*}
\sum_{\substack{u_m \in \B, \\ t_m < R^{-1} t_j^\epsilon}} \left|\left<|u_j|^2, 
u_m\right>\right|^2 
\ll 
\sum_{\substack{u_j \in \B, \\ t_m < R^{-1} t_j^\epsilon}} 
t_j^{-2+\epsilon}t_m^{-1+\epsilon} 
\ll 
R^{-2-\epsilon}t_j^{-2+\epsilon}.
\end{equation*}
Similarly, using (\ref{productformula3Dformula2}) and assuming GLH we get
\begin{equation*}
\left(\int_{|t| \leq R^{-1} t_j^{\epsilon}} \left<|u_j|^2, 
E\left(\cdot,1+it\right)\right>\langle \phi_R, 
E\left(\cdot,1+it\right)\rangle dt\right)^2 
\end{equation*}
is bounded by
\begin{equation*}
\begin{split} 
\|\f_R\|_2^2  \int_{|t| \leq R^{-1} t_j^{\epsilon}} 
\left|\left<|u_j|^2,E\left(\cdot,1+it\right)\right>\right|^2 dt  
& \ll 
\|\f_R\|_2^2 R^{-\epsilon} t_j^{-2+\epsilon} 
\int_{|t| \leq R^{-1} t_j^{\epsilon}} \frac{1}{1+|t|} dt \\
& \ll 
\|\f_R\|_2^2 R^{-\epsilon} t_j^{-2+\epsilon}. 
\end{split} 
\end{equation*}
Hence the contribution of the continuous spectrum is bounded by $ \|\f_R\|_2 
R^{-\epsilon} t_j^{-1+\epsilon} $ and we finally conclude 
\begin{equation*}
\left<\f, |u_j|^2\right> - \left<\f,1\right> = 
O_{M, \epsilon}\left(\|\phi\|_2 R^{-1-\epsilon} t_j^{-1+\epsilon} 
+ \|\phi\|_1 t_j^{-M}\right). 
\end{equation*}
This proves (\ref{propyoung3Dexplicit}).
\end{proof}

Now let us prove Part~$(a)$ of Theorem~\ref{theorem2}: 

\begin{proof}[Proof of Part~$(a)$ of Theorem~\ref{theorem2}] 
Taking $ \phi $ to be an approximation of the characteristic function 
of a ball of radius $ R $ and normalizing $ \left<\phi, 1\right> \asymp 
\vol(B_R) \asymp R^3 $ and $ \|\phi\|_2 \asymp R^{\frac{3}{2}} $ we get
\begin{equation*} \label{youngexplicitballsversiondim3-1} 
\frac{1}{\vol(B_R(Q))} \int_{B_R(Q)} |u_j(P)|^2 dv(P) = 
\frac{1}{\vol(\GmodHthree)} 
+ O_{\epsilon}\left(R^{-5/2-\epsilon} t_j^{-1+\epsilon}\right) 
+ O_{M}(t_j^{-M}).
\end{equation*}
The error term is smaller than the main term if and only if
\begin{align*}
R^{-5/2-\epsilon} t_j^{-1+\epsilon}=o(1),
\end{align*}
i.e. if and only if $R\gg t_j^{-\frac{2}{5}+\epsilon}$.
\end{proof}

\section{Subconvexity and mean Lindel\"of estimates}

For the study of QUE for Eisenstein series we need explicit subconvexity 
estimates for the $L$-functions appearing in the product formulae 
(\ref{tripleproducteisensteinseries3manifolds}) and (\ref{regularintegralkleinian}). 
In particular, we need good estimates for
\begin{align}  \label{lfunctionsthreedim}
L(1/2, u_j)^2 |L(1/2 +it, u_j)|^2
\end{align}
and
\begin{align}
\sum_{\substack{u_j\in\B, \\ t_j \leq R^{-1} T^{\epsilon}}} 
\frac{\left|L(1/2,u_j) \right|^2}{(1+|t_j|)^2}
\int_{T}^{2T}  \frac{\left|L(1/2+it,u_j) \right|^2}{(1+|t|)} \, dt.
\end{align}
Let 
\begin{equation*} 
C(t, t_j) := 
\left(1 + \left|t + \frac{t_j}{2}\right|^2\right) 
\left(1 + \left|t - \frac{t_j}{2}\right|^2\right) 
\end{equation*} 
be the analytic conductor for $ L(s, u_j) $ as defined in 
\cite{IwaniecSarnak}. Then we have the convexity bound 
\begin{equation*}
L(1/2 + it,  u_j) \ll_{\epsilon} C(t, t_j)^{1/4 + \epsilon}.
\end{equation*}
Petridis and Sarnak \cite{petridissarnak} were the first to prove 
a subconvexity result for $L(1/2 +2it, u_j)$ in the $ t $-aspect. 
For point-wise estimates we refer to the following recent hybrid 
estimate of Wu \cite[Cor.~1.6]{wu}: 

\begin{theorem}{\cite[Cor.~1.6]{wu}} 
Let $ \theta $ be the exponent towards the Ramanujan-Petersson 
conjecture. Then the following estimate holds: 
\begin{equation} \label{lfunctionsthreedimuniform}
L(1/2 + it, u_j) \ll_{\epsilon}  
C(t, t_j)^{\frac{1}{4} - \frac{1 - 2\theta}{32} 
+ \epsilon}. 
\end{equation}
\end{theorem}

By Nakasuji's work (see \cite[Cor.~1.2]{nakasuji}) we can take 
$ \theta = 7/64 $. Hence we can use the following unconditional 
upper bound 

\begin{equation} \label{lfunctionsthreedimuncond}
L(1/2 + it, u_j) \ll C(t, t_j)^{\frac{231}{1024} + \epsilon}.
\end{equation}

\subsection{Second moment estimates for central  $L $-values} 

For $\G = \pslzi$ we can improve the subconvexity exponent on 
average and get the following mean Lindel\"of estimate: 

\begin{proposition} \label{proplindelofdicrete} 
For $ \G = \pslzi $ we have: 
\begin{equation} \label{lfunctionsmean}
\sum_{t_j \leq T} |L(1/2, u_j)|^2 \ll T^{3+\epsilon}.
\end{equation}
\end{proposition}

\begin{proof}
The proof of Proposition~\ref{proplindelofdicrete} uses the 
standard approach of the approximate functional equation: 
the $ L $-function
\begin{align} \label{definitionlfunction}
L(s, u_j) = \sum_{n \in \mathbb{Z}[i] \setminus \{0\}} 
\frac{\lambda_j(n)}{N(n)^s} 
\end{align}
satisfies the functional equation
\begin{align}\label{eq:symfunc}
\gamma(s,t_{j}) L(s,u_{j}) = \gamma(1-s,t_{j}) L(1-s,u_{j}) 
\end{align}
where the gamma factor is given by 
\begin{equation*}
\gamma(s, t_{j}) = 
\pi^{-2s} \Gamma\left(s+\frac{it_{j}}{2}\right) 
\Gamma\left(s-\frac{it_{j}}{2}\right) 
\end{equation*}
(see e.g.\ \cite{heitkamp}, Satz~16.4, p.~115). Thus the 
approximate functional equation implies that we can write 
\begin{equation*} 
L(1/2, u_j) = \sum_{n} \frac{\lambda_j(n)}{\sqrt{N(n)}} \, 
V_{t_j}(N(n)) 
\end{equation*} 
where 
\begin{equation*} \label{approxsecond-1}
V_{t_j}(y) = 
\int_{(\sigma)} y^{-u} G(u) \frac{\gamma(1/2+u,t_j)}{\gamma(1/2,t_j)} 
\frac{du}{u}, 
\end{equation*}
$ G(u) $ is holomorphic, even and bounded on vertical strips 
and satisfies $ G(0) = 1 $. 
We adapt the proof of \cite[Thm.~3]{bccfl} for the central point 
$ s = 1/2 $ and as in the proof of \cite[Thm.~3]{bccfl} we use 
the mean Ramanujan bound due to Koyama~\cite[Thm.~2.1]{koyama2}
\begin{equation} \label{meanramanujan}
\sum_{N(n) < N} \abs{\lambda_{j}(n)}^{2} 
= 
O((1+\abs{t_{j}}^{\epsilon})N). 
\end{equation}
Note that we have the factor $ 1/\sqrt{N(n)} $ in the definition 
of the Hecke operators. This factor is not present in Koyama's 
definition. Then we can mostly follow the arguments of \cite{bccfl}. 
However, there is a difference concerning the estimation of the 
finite series part 
\begin{equation} \label{approxfirst}
\sum_{t_j \leq T} 
\Bigg|\sum_{N(n) \leq N} \frac{\lambda_j(n)}{\sqrt{N(n)}} \, 
V_{t_j}(N(n))\Bigg|^2. 
\end{equation}
The asymptotics of 
\begin{align} 
\frac{\gamma(1/2+u,t_j)}{\gamma(1/2,t_j)}
\end{align}
allow us to cut the sum (\ref{approxfirst}) to $ N(n) \approx t_j $. 
Thus we get
\begin{equation} \label{approxthird}
\sum_{t_j \leq T} |L(1/2, u_j)|^2  
\ll 
\sum_{t_j \leq T} 
\Bigg|\sum_{N(n) \leq N} \frac{\lambda_j(n)}{\sqrt{N(n)}}\Bigg|^2
\end{equation}
with $ N \approx T $. For $ \G = \pslzi $ the spectral large sieve 
of Watt \cite[Thm.~1]{watt} reads as 
\begin{align}\label{wattsieve}
\sum_{t_j \leq T} \Bigg|\sum_{N(n) \leq N} a_n \lambda_j(n)\Bigg|^2 
\ll 
\big(T^3 + T^{3/2} N^{1+\epsilon}\big) 
\sum_{N(n) \leq N} \left|a_n\right|^2 
\end{align}
which implies (\ref{lfunctionsmean}). 
\end{proof}

\begin{remark} 
Note that a difference between the situation of dimension~$2$ and dimension~$3$ 
is the available spectral large sieve. This is optimal for the case of the modular 
surface due to the work of Deshouillers and Iwaniec \cite{des-iwa}. 
For the Picard manifold $ \pslzi \setminus \mathbb{H}^3$ Watt's sieve~\cite{watt} 
is not optimal as the large sieve constant in (\ref{wattsieve}) is expected to be 
$T^3 + N^{1+\epsilon}$. This causes an important difference when bounding the 
$L$-functions attached to higher symmetric powers of $u_j$ (see \cite{bccfl}).
However, this does not affect the second moment upper bound in our case where we only 
have to consider bounds for $L(s, u_j)$. 
\end{remark}

\subsection{The integral second moment on the critical line} 

In this section we prove a mean Lindel\"of estimate for the integral 
second moment of $L(s, u_j)$ on the critical line as a consequence of 
the approximate functional equation.

\begin{proposition} \label{secondmomentlindelofprop}
For $ \G = \pslzi $ we have the following estimate for the integral 
second moment of $ L(s, u_j) $: 
\begin{equation} \label{secondmomentlindelof} 
\int_{T}^{2T} \left| L(1/2+it,u_j) \right|^2 dt 
\ll (1+T)^{1+\epsilon} (1+|t_j|)^{\epsilon}.
\end{equation}
\end{proposition}

\begin{proof}
Let $ G(u) $ be a function that is even, holomorphic and bounded for 
$ |\Re (u)| < 4 $ and satisfies $ G(0) = 1 $. 
The approximate functional equation for $ L(s, u_{j})$ implies 
that 
\begin{equation} \label{approxintegralmoment}
\begin{split} 
L(1/2+it, u_j) &= 
\sum_{n \in \mathbb{Z}[i] \setminus \{0\}} 
\frac{\lambda_j(n)}{N(n)^{\frac{1}{2}+it}} V_{t_j}(N(n),t) \\
& \quad 
+ \frac{\gamma(1/2-it, t_j)}{\gamma(1/2+it, t_j)} 
\sum_{n \in \mathbb{Z}[i] \setminus \{0\}} 
\frac{\lambda_j(n)}{N(n)^{\frac{1}{2}-it}} V_{t_j}(N(n),-t) 
\end{split} 
\end{equation}
where 
\begin{equation*} \label{approxsecond-2}
V_{t_j}(y,t) = 
\int_{(\sigma)} y^{-u} G(u) 
\frac{\gamma(1/2 + it + u,t_j)}{\gamma(1/2 + it, t_j)} \, \frac{du}{u}. 
\end{equation*}
Hence we infer that $ L(1/2+it, u_j) L(1/2-it, u_j) $ can be written 
as the sum of four terms that have the following form: 
\begin{equation} \label{approxintegralmoment2}
\sum_{n,m \neq 0} \frac{\lambda_j(n) \lambda_j(m)}{\sqrt{N(n)N(m)}} 
N(n)^{-it} N(m)^{\pm it} V_{t_j}(N(n), t) V_{t_j}(N(m),\mp t). 
\end{equation}
We bound the first of them, the remaining terms can be treated in the 
same way. We want to estimate
\begin{equation*}
\int_{T}^{2T} \left(\frac{N(m)}{N(n)}\right)^{it} V_{t_j}(N(n), t) 
V_{t_j}(N(m), -t) dt.
\end{equation*}
Estimating the $\G$-factors we can shift the integral to $\sigma_1 
= \sigma_2 = \epsilon$. Since we only care for a crude bound without 
any explicit asymptotics we bound 
\begin{equation*} 
\int_{T}^{2T} \left| L(1/2+it,u_j) \right|^2 dt 
\end{equation*}
using (\ref{meanramanujan}) by 
\begin{equation*}
\sum_{n \neq 0} \frac{|\lambda_j(n)|^2}{N(n)} 
\int_{T}^{2T} \frac{|t|^{4 \epsilon}}{N(n)^{\epsilon}} dt 
\ll 
(1+T)^{1+4\epsilon} (1+|t_j|)^{\epsilon}
\end{equation*}
and the statement follows.
\end{proof}

Note that, using his trace formula, Kuznetsov \cite{kuznetsov} gave an 
explicit asymptotic result for the second moment in the case of the modular 
group with an error term of $ O(c_j T + T^{6/7 + \epsilon}) $ for some 
constant $c_j$. In our case the crude bound (\ref{secondmomentlindelof}) 
suffices.

\section{Equidistribution of Eisenstein series on $ \Gamma \setminus 
\mathbb{H}^3$} 

In this section we prove parts $ (b) $ and $ (c) $ of Theorem~\ref{theorem2}, i.e.\ the 
unconditional equidistribution of Eisenstein series in shrinking balls. It 
follows from  \cite{michelvenkatesh} that Theorem~\ref{michelvenkateshprop} 
holds in some generality. In particular, it allows us to estimate the spectral 
expansion of
\begin{equation*}
\frac{1}{\vol(B_R (P))} \int_{B_R(P)} |E(Q, 1+it)|^2 dv(Q) 
- \frac{|\mathcal{O}_K^{*}| \sqrt{|d_K|} \log (1+t^2)}{4 \vol(\GmodHthree)} 
\end{equation*}
for $ F = E(P, s_1) E(P,s_2) $ and $ G = \phi_R $. For the Eisenstein series 
we get the following result: 

\begin{proposition} \label{propyoungMaass3Deisenstein} 
Let $ \Gamma $ be a Bianchi group of class number one and $ \phi = \phi_R $ 
be a fixed test function as in (\ref{testphiproperty3}). Assume that $ R \gg 
t_j^{-\delta} $ for some fixed $ 0 < \delta < 1 $. Then we have:  
\begin{equation} \label{eisensteinspecexp}
\begin{split} 
\int_{\GmodHthree} \phi_R(P) \left|E(P, 1+it) \right|^2 dv(P) 
&= 
\log(1+t^2) 
\frac{|\mathcal{O}_K^{*}| \sqrt{|d_K|}}{4 \vol(\GmodHthree)} 
\int_{\GmodHthree} \ \phi_R(P) dv(P) \\
& \quad 
+ O_{\epsilon}\left(\|\phi\|_2 (1+|t|)^{-\frac{1-2\theta}{8}+\epsilon} 
R^{-\frac{11 + 2\theta}{4} - \epsilon}\right) \\
& \quad 
+ O_{\epsilon} \left(\|\phi\|_2 (1+|t|)^{-1/3 + \epsilon} 
R^{-\frac{14+3a}{18}-\epsilon}\right) \\
& \quad 
+
O\left(\|\phi \|_1 \frac{\log t}{\log \log t}\right) 
+ O\left(\|\phi\|_1 t^{-2}\right) 
+ O_M\left(\|\phi\|_2 t^{-M}\right) 
\end{split} 
\end{equation}
for any $M \gg 1$. Here $ \theta $ is an exponent towards the Ramanujan--Petersson conjecture and 
the parameter $ 0 \leq a \leq 1 $ is related to the subconvexity exponent for 
the twelfth moment of the Riemann zeta function as in (\ref{12thmomwitha}). 

\bigskip 

\noindent 
For $\Gamma = \pslzi$ the first error term can be improved to
\begin{equation} \label{eisensteinspecexppicardgroup}
O_{\epsilon}\left(\|\phi\|_2 (1+|t|)^{-\frac{1- 2\theta}{8} +\epsilon} 
R^{-\frac{15+ 2\theta}{8} -\epsilon}\right).
\end{equation}
\end{proposition} 

\begin{remark} 
If we assume the Generalized Lindel\"of Hypothesis, then for any 
$ \Gamma = \textup{PSL}_2(\mathcal{O}_K) $, with $ K $ an imaginary 
quadratic number field  of class number one, the first error term 
appearing in (\ref{eisensteinspecexp}) is bounded by
\begin{equation*}
O_{\epsilon}\left(\|\phi\|_2 (1+|t|)^{-1+\epsilon} R^{-1-\epsilon}\right) 
\end{equation*}
and the second term by
\begin{equation*}
O_{\epsilon}\left(\|\phi\|_2 (1+|t|)^{-1+\epsilon} R^{-\epsilon}\right).
\end{equation*}
\end{remark} 

We split the proof of Proposition~\ref{propyoungMaass3Deisenstein} in 
three parts: first we determine the contribution of the various elements 
coming from the spectral decomposition (cf.~Theorem~\ref{michelvenkateshprop}) 
and in the end we combine these results to prove the proposition.

\subsection{The contributions of $ \lambda_0 = 0 $ and $  \langle \mathcal{E}, 
\phi_R \rangle $} \label{zeroeigenvalue}

The main term of 
\begin{equation*}
\left<\left| E(\cdot, 1+it)\right|^2, \phi_R\right>
\end{equation*}
was already determined by Koyama~\cite[p.~485]{koyama1} who also gave an 
error term. However, there is an error in his proof as he does not take 
the term coming from the double pole of $ \zeta_K^2(s/2) $ at $ s = 1 $ 
into account. This error has subsequently been corrected by 
Laaksonen~\cite[Rem.~1]{laaksonen}. 
Using the spectral theorem and Zagier's Rankin--Selberg method of functions 
that are not of rapid decay we see that the main term of the asymptotic 
behaviour comes from the following terms in (\ref{promichelvenk}): 
\begin{equation*}
\left<|E(P, 1+it)|^2, u_0\right>_{\rn} \left<u_0, \phi_R\right> 
+  \langle \mathcal{E}, \phi_R \rangle
\end{equation*}
More precisely, as $ \left<|E(P, 1+it)|^2, 1\right>_{\rn} $ vanishes the 
term $ \langle \mathcal{E}, \phi_R \rangle $ is responsible for the main 
term. For determining its asymptotic behaviour we can adapt the approach 
of \cite[pp.~976, 980]{young}. First of all, we note that the divergence causing 
term of the product $ E(P, 1+s_1) E_(P, 1+s_2) $ is given by 
\begin{equation*} 
\frac{|\mathcal{O}_K^*|^2}{4} 
\left(r^{2+s_1+s_2} + \phi(s_2) r^{2+s_1-s_2} + \phi(s_1) r^{2-s_1+s_2} 
+ \phi(s_1) \phi(s_2) r^{2-s_1-s_2}\right). 
\end{equation*} 
Eventually, we will set $ s_1 = \alpha + it $ and $ s_2 = -it $ and consider 
the limit $ \alpha \rightarrow 0 $. Thus we have 
\begin{equation*} 
\begin{split} 
\mathcal{E} &= 
\frac{|\mathcal{O}_K^*|}{2} 
\big(E(P, 2+s_1+s_2) + \phi(s_2) E(P, 2+s_1-s_2)  \\ 
& \qquad \phantom{\frac{|\mathcal{O}_K^*}{2} \big(} 
+ \phi(s_1) E(P, 2-s_1+s_2) 
+ \phi(s_1) \phi(s_2) E(P, 2-s_1-s_2)\big).
\end{split} 
\end{equation*} 
The second and the third term appearing on the right-hand side do not 
contribute to the main term. Using the decaying properties of $ \phi_R $ 
as in \cite{young}, p.~980 their contribution can be bounded by 
\begin{equation*} 
\langle E(P, 2 \pm s_1 \mp s_2), \phi_R \rangle = O\left(t^{-M}\right). 
\end{equation*} 
Hence it remains to treat 
\begin{equation*} \label{eq:zwischenergebnis} 
\frac{|\mathcal{O}_K^*|}{2} \lim_{\alpha \rightarrow 0} 
\langle E(P, 2+\alpha) + \phi(\alpha + it) \phi(-it) E(P, 2-\alpha), 
\phi_R\rangle. 
\end{equation*} 
In order to determine this limit note that $ \phi(s) \phi(-s) = 1 $ 
and that 
\begin{equation*} 
\textup{res}_{s = 2} E(P, s) = \frac{2 \pi^2}{|d_K| \zeta_K(2)} 
= \frac{\sqrt{|d_K|}}{2 \vol(\GmodHthree)}
\end{equation*} 
(see \cite{elstrodt}, Theorem~1.11, pp.~243--244 and \cite{elstrodt}, 
Theorem~1.1, p.~312). This implies that 
\begin{equation*}
E(P, 2+\alpha) + \phi(\alpha + it) \phi(-it) E(P, 2-\alpha) 
= 
2a(P) 
-  
\frac{\phi'}{\phi}(it) \, \textup{res}_{s = 2} E(P, s) 
+ O(\alpha) 
\end{equation*} 
where $ a(P) $ denotes the constant term in the Laurent expansion of 
$ E(P, 2 + \alpha) $ about $ \alpha = 0 $. Using the explicit form 
of the scattering matrix (\ref{eq:scattering_H3}) we obtain 
\begin{equation*} 
\begin{split} 
\frac{\phi'}{\phi}(it) = 
2 \log\left(\frac{2 \pi}{\sqrt{|d_K|}}\right) - \frac{\Gamma'}{\Gamma}(1+it) 
- \frac{\Gamma'}{\Gamma}(1-it) - \frac{\zeta_K'}{\zeta_K}(1+it) 
- \frac{\zeta_K'}{\zeta_K}(1-it) 
\end{split} 
\end{equation*} 
(cf.\ also \cite{raulf}, p.~132). The main term in the asymptotics comes 
from the terms involving the $ \Gamma $ function whereas the other terms 
are absorbed in the error term. Using the estimate of \cite[p.~485]{koyama1} 
\begin{equation*}
\frac{\zeta_{K}^{'}(1\pm it)}{\zeta_K (1\pm it)} \ll \frac{\log t}{ \log \log t},
\end{equation*} 
for the logarithmic derivative of the Dedekind zeta function we obtain that 
its contribution to the error term is 
\begin{equation*}
O\left(\|\phi_R\|_1 \frac{\log t}{\log \log t}\right).
\end{equation*} 
Furthermore, Stirling's formula implies that 
\begin{equation*} 
\frac{\Gamma'}{\Gamma}(1+it) + \frac{\Gamma'}{\Gamma}(1-it) = 
\log\left(1 + t^2\right) + O\left(t^{-2}\right). 
\end{equation*} 
Thus we infer 
\begin{equation*} 
\begin{split} 
\langle \mathcal{E}, \phi_R\rangle 
&= 
\frac{\pi^2 |\mathcal{O}_K^*|}{|d_K| \zeta_K(2)} \log(1+t^2) \langle 1, \phi_R\rangle 
+ O\left(t^{-2} \|\phi_R\|_1\right) 
+ O\left(\|\phi_R\|_1 \frac{\log t}{\log \log t}\right). 
\end{split} 
\end{equation*}

\subsection{The contribution of the discrete spectrum} \label{section:discrete_H3} 

In this section we determine the contribution of the discrete spectrum. 
Using (\ref{tripleproducteisensteinseries3manifolds}), Bessel's inequality 
and (\ref{testphiproperty3}), the contribution of the discrete spectrum 
$ u_j \in \B $ in the spectral expansion of (\ref{eisensteinspecexp}) is 
bounded as follows: 
\begin{equation} \label{mainterm94}
\begin{split} 
\sum_{u_j \in \mathcal{B}} & \langle |E(\cdot, 1 + it)|^2, u_j\rangle 
\langle u_j, \phi_R \rangle \\ 
& \ll 
\|\phi_R\|_2 \left(\sum_{t_j \leq R^{-1} t^{\epsilon}}  
\left| \left<\left|E(\cdot, 1+it)\right|^2, u_j\right> \right|^2\right)^{1/2} 
+O_M\left(\|\phi_R\|_2 t^{-M}\right) \\
& \ll 
\|\phi_R\|_2 \left(\sum_{t_j \leq R^{-1} t^{\epsilon}} 
\frac{\left|L(1/2,u_j) \right|^2 \left|L(1/2+it,u_j)\right|^2}{(1+|t_j|) 
(1+|t|)^2} \right)^{1/2} + O_M\left(\|\phi_R\|_2 t^{-M}\right).
\end{split} 
\end{equation}
Let us now consider the case that $ \Gamma = \pslzi $. In this case, 
Proposition~\ref{proplindelofdicrete} and summation by parts imply 
\begin{equation*}
\sum_{t_j \leq R^{-1} t^{\epsilon}} 
\frac{\left|L(1/2,u_j)\right|^2}{(1+|t_j|)^B} 
\ll R^{-3+B} t^{\epsilon}.
\end{equation*}
Using the point-wise estimate (\ref{lfunctionsthreedimuniform}) we bound the 
first term on the right-hand side of (\ref{mainterm94}) by
\begin{equation} \label{boundbothsubconvexities}
\|\phi_R\|_2 (1+|t|)^{-\frac{1-2\theta}{8} +\epsilon} 
\left(\sum_{t_j \leq R^{-1} t^{\epsilon}}  
\frac{\left|L(1/2,u_j)\right|^2}{(1+|t_j|)^{-\left(\frac{3+2\theta}{4} 
+ \epsilon\right)}}\right)^{1/2} 
\ll 
\|\phi_R\|_2 (1+|t|)^{-\frac{1-2\theta}{8} +\epsilon}  
R^{-\frac{15 + 2 \theta}{8} -\epsilon}.
\end{equation}
For $\Gamma$ different to PSL$_2(\Z[i]) $ we use the point-wise estimate 
(\ref{lfunctionsthreedimuniform}) also for the central $ L $-value 
$ L(1/2,u_j) $ in (\ref{boundbothsubconvexities}) and the Weyl law to 
get the upper bound 
\begin{equation*} 
\|\phi_R\|_2 (1+|t|)^{-\frac{1-2\theta}{8} +\epsilon} 
R^{-\frac{11+ 2\theta}{4} +\epsilon}.
\end{equation*}

\begin{remark}
We mention that the large sieve of Watt (\ref{wattsieve}) and consequently
Proposition~\ref{proplindelofdicrete} should hold for any $ \Gamma = 
\textup{PSL}_2(\mathcal{O}_K) $ with $ \mathcal{O}_K $ being the ring of 
integers of an imaginary quadratic number field of class number one. However, 
we treat the two cases separately to show that Proposition~\ref{proplindelofdicrete} 
is not necessary to obtain equidistribution results in some range. In particular, 
our exponents can be improved if one can prove better subconvexity results either for 
the discrete spectrum or for the continuous part. 
\end{remark}

\subsection{The contribution of the continuous spectrum} 
\label{section:continuous_H3}

In this section we estimate the the contribution of the continuous 
spectrum to the asymptotics. Using the properties of $ \phi_R $ 
(see (\ref{testphiproperty3})) as in \cite{young}, (4.9), p.~975 
and Lemma~\ref{eisensteinproduct3dim} we obtain that 
\begin{equation} \label{eq:Eisenstein_H3} 
\begin{split} 
\int_{-\infty}^{\infty} & 
\left<|E(\cdot, 1+it)|^2 , E(\cdot, 1+it')\right>_{\rn} 
\left<E(\cdot, 1+it'), \phi_R\right> dt' \\ 
&= 
\int_{|t'| \leq R^{-1} t^{\epsilon}} 
\left<|E(\cdot, 1+it)|^2 , E(\cdot, 1+it')\right>_{\rn} 
\left<E(\cdot, 1+it'), \phi_R\right> dt' \\ 
& \quad 
+ 
O\left(\|\phi_R\|_2 t^{-M}\right) \\ 
& \ll 
\|\phi_R\|_2 \left(\int_{|t'| \leq R^{-1} t^{\epsilon}} 
\left|\left<|E(\cdot, 1+it)|^2 , E(\cdot, 1+it')\right>_{\rn}\right|^2 
dt'\right)^{1/2}  
+ \|\phi_R\|_2 t^{-M} 
\end{split} 
\end{equation} 
where we applied the Cauchy--Schwarz inequality and \cite{elstrodt}, (3.13), 
p.~268. Thus it remains to estimate 
\begin{equation} \label{def:I}
I(t) := 
\int_{|t'| \leq R^{-1} t^{\epsilon}} 
\left|\left<|E(\cdot, 1+it)|^2 , E(\cdot, 1+it')\right>_{\rn}\right|^2 dt'. 
\end{equation} 
By Lemma~\ref{eisensteinproduct3dim} and standard bounds on the Dedekind 
zeta function we infer 
\begin{equation} \label{eisensteinntegral3dimpointwise} 
\begin{split} 
I(t) 
& \ll 
\frac{R^{-\epsilon}}{(1+|t|)^{2-\epsilon}} 
\int_{|t'| \leq R^{-1} t^{\epsilon}} 
\frac{\left|\zeta_K \left(\frac{1 + it'}{2}\right)\right|^4 
\left|\zeta_K\left(\frac{1 + it'}{2} - it\right)\right|^2 
\left|\zeta_K\left(\frac{1+it'}{2} + it\right)\right|^2}{1+|t'|} \, dt'. 
\end{split} 
\end{equation} 
If we were to use only Heath-Brown's Weyl bound (see \cite{heathbrown2}) 
for the Dedekind zeta function 
\begin{equation} \label{weylsboundfornumberfield}
\zeta_{K}(1/2+it) \ll_K (1+|t|)^{1/3+\epsilon}, 
\end{equation}
then we would get a worse bound for the contribution of the continuous 
spectrum than for the contribution of the discrete spectrum. 
Although the Weyl bound suffices to give a good estimate in the case of
the modular surface, in our case we need something better than this. 
We can do slightly better using moment estimates for 
the Dedekind zeta function. 
As, so far, there is no mean Lindel\"of bound available for the fourth moment 
of $ \zeta_K $ known for any class number one field $ K $ we proceed as follows: 
we use the well-known identity 
\begin{equation*} 
\zeta_{\mathbb{Q}\left(\sqrt{D}\right)} = \zeta(s) L\left(s, \chi_{|D|}\right)
\end{equation*}
and H\"older's inequality to get 
\begin{equation} \label{fourthmomentofdedekindzeta}
\begin{split} 
\int_{-T}^{T} & \left|\zeta_{\mathbb{Q}(\sqrt{D})}(1/2+it)\right|^{4} dt \\ 
& \quad 
\ll 
\left(\int_{-T}^{T}  \left|\zeta (1/2+it)\right|^{12}  dt\right)^{1/3}
\left(\int_{-T}^{T}  \left|L\left(1/2+it, \chi_{|D|}\right)\right|^{6}  
dt\right)^{2/3}.
\end{split} 
\end{equation}
First we bound the second integral appearing on the right-hand side 
of (\ref{fourthmomentofdedekindzeta}). Unfortunately, the known results 
for the sixth moment of Dirichlet $ L $-functions (see \cite{ConreyIwaniec2000} 
and \cite{PetrowYoung18}) do not give a sufficiently good upper bound 
for our application. 
Instead we bound two of the six $ L $-factors point-wise by a Weyl 
bound (see \cite{PetrowYoung18}, (1.1)) and use a mean Lindel\"of 
result for the \textit{fourth} moment of our Dirichlet $ L $-function 
(see e.g.\ \cite{PetrowYoung18}, (1.8)) so that we obtain 
\begin{equation*} 
\int_{-T}^{T}  \left|L\left(1/2+it, \chi_{|D|}\right)\right|^{6} dt 
\ll_D 
T^{4/3 + \epsilon}.
\end{equation*} 
In order to bound the first term on the right-hand side of 
(\ref{fourthmomentofdedekindzeta}) we note that Heath-Brown's 
bound~\cite{heathbrown} for the twelfth moment of the Riemann zeta 
function
\begin{equation} \label{hbrown12th}
\int_{-T}^{T} |\zeta(1/2+it)|^{12}  dt \ll T^{2+\epsilon}
\end{equation}
implies $ \zeta(1/2+it) \ll t^{1/12+\epsilon} $ on average. Assuming 
the improved bound 
\begin{equation} \label{12thmomwitha}
\int_{-T}^{T} |\zeta(1/2+it)|^{12}  dt \ll T^{1+a+\epsilon}, 
\end{equation}
$ a \in [0, 1] $, we finally obtain the bound 
\begin{equation} \label{fourthmomentzetakgeneral}
\int_{-T}^{T} \left|\zeta_{\mathbb{Q}(\sqrt{D})} (1/2+it)\right|^{4} dt 
\ll T^{\frac{11}{9} + \frac{a}{3} + \epsilon}. 
\end{equation}
If we replace in the integral (\ref{eisensteinntegral3dimpointwise}) 
the last two Dedekind zeta functions by the Weyl bound 
(\ref{weylsboundfornumberfield}),  this bound implies 
\begin{equation} \label{holder} 
\begin{split} 
I(t) 
& \ll 
\frac{R^{-\epsilon}}{(1+|t|)^{2/3-\epsilon}} 
\int_{|t'| \leq R^{-1} t^{\epsilon}} 
\left|\zeta_K \left(\frac{1 + it'}{2}\right)\right|^4 
(1 + |t'|)^{1/3 + \epsilon} \, dt' \\ 
& \ll 
R^{-\frac{14+3a}{9} + \epsilon} (1+|t|)^{-2/3+\epsilon}. 
\end{split} 
\end{equation} 
Thus the contribution of the continuous spectrum in 
(\ref{eisensteinspecexp}) is bounded by  
\begin{equation} \label{contrgeneralk} 
\|\phi\|_2 
R^{-\frac{14+3a}{18} + \epsilon} (1+|t|)^{-1/3+\epsilon}
+ \|\phi_R\|_2 t^{-M} 
\end{equation} 
This concludes the proof of Proposition~\ref{propyoungMaass3Deisenstein}. 

\begin{remark}
The fourth moment for the Dedekind zeta function of $K=\mathbb{Q}(i)$ 
was extensively studied in the deep work of Bruggeman and Motohashi \cite{bruggeman}. 
Nevertheless, even in this case the mean Lindel\"of bound 
for the fourth moment of $\zeta_K$ remains out of reach. 
\end{remark}

\subsection{Proof of Theorem \ref{theorem2}~(b) and (c)} 
We now combine our previous results to prove parts (b)  and (c) of Theorem~\ref{theorem2}. 
Let $ \phi $ approximate the characteristic function of a ball of radius 
$ R $ and normalize it such that $ \left<\phi, 1\right> \asymp R^3 $. For 
$ K \not= \mathbb{Q}(i) $ it follows from $ \|\phi\|_2 \asymp 
R^{\frac{3}{2}} $ and Proposition~\ref{propyoungMaass3Deisenstein} that 
\begin{equation} \label{youngexplicitballsversiondim3-2} 
\begin{split} 
\frac{1}{\vol(B_R (P))} & \int_{B_R(P)} |E(Q, 1+it)|^2 dv(Q) 
- \frac{|\mathcal{O}_K^{*}| \sqrt{|d_K|} \log(1+t^2)}{4 \vol(\GmodHthree)} \\ 
&= 
O_{\epsilon}\left((1+|t|)^{-\frac{1-2\theta}{8} + \epsilon}  
R^{-\frac{17 + 2\theta}{4} - \epsilon}\right) 
+ O_{\epsilon} \left((1+|t|)^{-1/3 + \epsilon} 
R^{-\frac{41+3a}{18} - \epsilon}\right) \\ 
& \quad 
+ O\left(\frac{\log t}{\log \log t}\right) + O\left(t^{-2}\right).
\end{split} 
\end{equation}
Thus we have equidistribution up to
\begin{equation*}
R \gg 
\max \left\{t^{-\frac{1-2\theta}{34+4\theta}+\epsilon}, 
t^{-\frac{6}{41+3a}+\epsilon}\right\} = 
t^{-\frac{1-2\theta}{34+4\theta}+\epsilon}.
\end{equation*}
For $ \theta = 7/64 $ this exponent equals $25/1102 $. 

\bigskip 

If $ K = \mathbb{Q}(i) $ then the error term coming from the discrete 
spectrum can be improved. This corresponds to the first error term on 
the right-hand side of (\ref{youngexplicitballsversiondim3-2}) which 
can be replaced by 
\begin{equation*}
O_{\epsilon}\left((1+|t|)^{-\frac{1-2\theta}{8} +\epsilon} 
R^{-\frac{27+2\theta}{8} -\epsilon}  \right). 
\end{equation*}
Consequently, we see that equidistribution holds up to 
\begin{equation*}
R \gg \max
\left\{t^{-\frac{1-2\theta}{27+2\theta}+\epsilon},  
t^{-\frac{6}{41+3a}+\epsilon}\right\} 
= t^{-\frac{1-2\theta}{27+2\theta}+\epsilon}.
\end{equation*}
Using the best known result for the Ramanujan--Peterson conjectures 
gives us the exponent $ 25/871 $. 

\bigskip 

If we assume the Generalized Lindel\"of Hypothesis for the 
$ L $-functions appearing in (\ref{mainterm94}) and 
(\ref{eisensteinntegral3dimpointwise}) we get for the contribution 
of the discrete spectrum the improved bound 
\begin{equation*} 
\|\phi\|_2  \left(\sum_{t_j \leq R^{-1} t^{\epsilon}} \frac{|t + 
t_j|^{\epsilon}}{(1+|t_j|)(1+|t|)^2}\right)^{1/2} 
\ll \|\phi\|_2 (1+|t|)^{-1+\epsilon} R^{-1-\epsilon} 
\end{equation*}
and for the contribution of the continuous spectrum the improved 
bound 
\begin{equation*} 
\frac{\|\phi\|_2 R^{-\epsilon}}{(1+|t|)^{1-\epsilon}} 
\left( \int_{|t'| \leq R^{-1} t^{\epsilon}} 
\frac{|t'|^{\epsilon}}{1+|t'|} \, dt'\right)^{1/2} 
\ll 
\frac{\|\phi\|_2 R^{-\epsilon}}{(1+|t|)^{1-\epsilon}}.
\end{equation*}
Thus the discrete contribution is $ o(1) $ if $ R \gg t^{-2/5 
+ \epsilon} $ whereas the continuous contribution is $ o(1) $ 
if $ R \gg t^{-2/3 + \epsilon} $.

\section{Quantum variance of Eisenstein series for shrinking balls 
on $ \GmodHthree $} \label{sectionquantumvareisenthreedim}

In this section we prove Theorem~\ref{theoremeisend3}. Note that for 
$ \Gamma = \pslzi $ we have $ |\mathcal{O}_K^{*}| = 4 $ and $ d_K = 4 $. 
To bound the second moment of
\begin{equation}\label{varianceeisenstein3d}
\frac{1}{\log (1+t^2) \vol(B_R (P))} \int_{B_R(P)} |E(Q, 1+it)|^2 dv(Q) 
- \frac{2}{\vol (\GmodHthree)}   
\end{equation}
we use Theorem~\ref{michelvenkateshprop} in the version for the hyperbolic 
3-space as in the proof of Proposition~\ref{propyoungMaass3Deisenstein}. 
For this we chose $ \phi_R $ as in (\ref{testphiproperty3}).

\subsection{The contribution of $ \lambda_0 $ and $ \left< \mathcal{E}, 
\phi_R \right> $} 

It follows from Section~\ref{zeroeigenvalue} that the contribution of the 
eigenvalue $ \lambda_0 = 0 $ and of $ \left< \mathcal{E}, \phi_R\right> $ 
in \ref{varianceeisenstein3d} is
\begin{equation*}
O\left(\frac{1}{\log \log t}\right).
\end{equation*}

\subsection{The contribution of the discrete spectrum}

Using (\ref{tripleproducteisensteinseries3manifolds}) as in 
Section~\ref{section:discrete_H3}, (\ref{mainterm94}), we 
infer that the contribution of the discrete spectrum to 
(\ref{varianceeisenstein3d}) is majorized by
\begin{equation*}
\frac{1}{R^{3/2}} \left(\sum_{\substack{u_j \in \B, \\ t_j \leq 
R^{-1} t^{\epsilon}}} \frac{\left|L(1/2, u_j)\right|^2 
\left|L(1/2 + it, u_j)\right|^2}{(1+|t_j|) (1+|t|)^2}\right)^{1/2} 
+ O_M\left(R^{-3/2} t^{-M}\right).
\end{equation*}
Keeping in mind that Proposition~\ref{secondmomentlindelofprop} implies 
\begin{equation*}
\int_{T}^{2T} \frac{\left|L(1/2+it,u_j) \right|^2}{(1+|t|)^2} \, dt 
\ll \frac{1}{T^2} \int_{T}^{2T} \left|L(1/2+it,u_j) \right|^2  \, dt 
\ll (1+ T)^{-1+\epsilon} (1+|t_j|)^{\epsilon} 
\end{equation*} 
we therefore infer that the contribution of the discrete spectrum to 
the left-hand side of (\ref{eisenvariancebound}) is bounded by 
\begin{equation} \label{ineq:1} 
\begin{split} 
M_D(R, T) & \coloneqq 
\int_{T}^{2T} \frac{1}{R^3} 
\sum_{\substack{u_j \in \B, \\ t_j \leq R^{-1} t^{\epsilon}}} 
\frac{\left|L(1/2,u_j)\right|^2 \left|L(1/2+it,u_j)\right|^2}{(1+|t_j|) 
(1+|t|)^2} \, dt \\
& \ll 
\frac{1}{R^3} \sum_{\substack{u_j \in \B, \\ t_j \leq R^{-1} (2T)^{\epsilon}}} 
\frac{\left| L(1/2,u_j) \right|^2}{(1+|t_j|)}
\int_{T}^{2T}  \frac{\left| L(1/2+it,u_j) \right|^2}{(1+|t|)^2} \, dt \\ 
& \ll 
\frac{T^{-1+\epsilon}}{R^3} 
\sum_{\substack{u_j \in \B, \\ t_j \leq R^{-1} (2T)^{\epsilon}}} 
\frac{\left|L(1/2,u_j)\right|^2}{(1+|t_j|)^{1-\epsilon}}. 
\end{split} 
\end{equation}
The last sum on the right-hand side of (\ref{ineq:1}) is estimated 
using partial summation and the mean-subconvexity estimate 
(\ref{lfunctionsmean}) so that 
\begin{equation*}
M_D (R, T) \ll R^{-5-\epsilon} T^{-1+\epsilon}. 
\end{equation*}

\subsection{The contribution of the continuous spectrum}

In order to estimate the contribution of the continuous spectrum to 
the error term in (\ref{eisenvariancebound}) we adapt the approach 
of Section~\ref{sectionquantumvareisen3} to the situation of the 
hyperbolic 3-space. By the same arguments as in the proof of 
Proposition~\ref{propyoungMaass3Deisenstein} (see 
Section~\ref{section:continuous_H3}) we infer that
\begin{equation} \label{promichelvenk2} 
\begin{split} 
M_C(R, T) & \coloneqq 
\int_{T}^{2T} \left| 
\frac{1}{R^3} \int_{-\infty}^{\infty} 
\left<|E(\cdot, 1+it)|^2 , E(\cdot, 1+it')\right>_{\rn} 
\left<E(\cdot, 1+it'), \phi_R\right> dt'\right|^2 dt \\ 
& \ll 
\frac{\|\phi_R\|_2^2}{R^6} \int_{T}^{2T} I(t) \, dt 
\end{split} 
\end{equation} 
where $ I(t) $ is defined in (\ref{def:I}). By (\ref{eisensteinntegral3dimpointwise}) 
and the Cauchy--Schwarz inequality we therefore obtain

\begin{equation*} \label{promichelvenkexpanded2} 
\begin{split} 
& M_C(R, T) \\ 
& \ll 
\frac{1}{R^{3 + \epsilon}} 
\int_{T}^{2T}
\int_{|t'| \leq R^{-1} t^{\epsilon}} 
\frac{\left|\zeta_K\left(\frac{1 + it'}{2}\right)\right|^4 
\left|\zeta_K\left(\frac{1 + it'}{2} - it\right)\right|^2 
\left|\zeta_K\left(\frac{1 + it'}{2} + it\right)\right|^2}{(1+|t'|) 
(1+|t|)^{2-\epsilon}} \, dt' \, dt \\ 
& \ll 
\frac{T^{-2 + \epsilon}}{R^{3 + \epsilon}} 
\int_{|t'| \leq R^{-1} T^{\epsilon}} \frac{\left|\zeta_K\left(\frac{1 
+ it'}{2}\right)\right|^4}{(1+|t'|)} 
\int_{T}^{2T} 
\left|\zeta_K\left(\frac{1 + it'}{2} - it\right)\right|^2 
\left|\zeta_K\left(\frac{1 + it'}{2} + it\right)\right|^2 \, dt \, dt'.
\end{split} 
\end{equation*} 
The inner integral can be evaluated using the Cauchy--Schwarz inequality and 
the estimate for the fourth moment of the Dedekind zeta function 
(\ref{fourthmomentzetakgeneral}). We conclude
\begin{equation*}
\begin{split} 
M_C(R, T)  \ll 
\frac{T^{-2 + \epsilon}}{R^{3 + \epsilon}} 
\int_{|t'| \leq R^{-1} T^{\epsilon}} 
\frac{|\zeta_K\left(\frac{1 + it'}{2}\right)|^4}{(1+|t'|)} \, 
(T+t')^{\frac{11}{9}+ \frac{a}{3} + \epsilon} \, dt'. 
\end{split} 
\end{equation*} 
Using that $t' \leq R^{-1} T^{\epsilon} \ll T$ and (\ref{fourthmomentzetakgeneral}) again we get
\begin{equation*} 
\begin{split} 
M_C(R, T) \ll \frac{T^{-\frac{7}{9} + \frac{a}{3} 
+ \epsilon}}{R^{\frac{29}{9} + \frac{a}{3} + \epsilon}}.
\end{split} 
\end{equation*}

\subsection{Proof of Theorem~\ref{theoremeisend3}} 

To prove Theorem~\ref{theoremeisend3} we combine our previous results 
for the contribution of the discrete and the continuous spectrum and 
infer 

\begin{equation*}
\int_{T}^{2T} 
\left|\frac{\int_{B_R(P)} |E(Q,1+it)|^2 dv(Q) }{\log(1+t^2) \vol(B_R(P))} 
- \frac{2}{\vol(\GmodHthree)}\right|^2 dt 
\ll \frac{T^{-1+\epsilon}}{R^{5+\epsilon}} +\frac{T^{-7/9 + a/3 + 
\epsilon}}{R^{29/9 + a/3 + \epsilon}} + \frac{T}{(\log \log T)^2}.
\end{equation*}
This is $o(T)$ if and only if   
\begin{equation*}
R \gg \max \left\{T^{-2/5+\epsilon},  T^{-\frac{16-3a}{29+3a} + \epsilon}\right\}.
\end{equation*}
As $ f(a) = -\frac{16-3a}{29+3a} $ is increasing, $ f(a) \leq f(1) = -13/32 
< -2/5 $ and $ T^{-\frac{16-3a}{29+3a} + \epsilon} \leq T^{-2/5+\epsilon} $ 
for $ a \in [0, 1] $. Thus Heath-Brown's result for the twelfth moment of the 
Riemann zeta function suffices to prove that quantum ergodicity holds for $ R 
\gg T^{-2/5+\epsilon} \asymp t^{-2/5 +\epsilon} $.

\section{Massive irregularities in shrinking balls of large dimension} 
\label{sectionlargedimensions}

In this final section we generalize Theorem~\ref{proposition1} to compact 
arithmetic quotients of the $ n $-dimensional hyperbolic space for $ n \geq 
4 $. Our result follows from an explicit estimate of the Selberg transform 
for the characteristic kernel and a lower bound of Donnelly~\cite{donnelly} 
for the sup-norm of Laplace eigenfunctions on $ n $-dimensional compact 
arithmetic manifolds. Let us write $ \lambda_j = (n-1)^2/4 + t_j^2 $ for 
the Laplace eigenvalue of the cusp form $ \phi_j $ and we will always assume $t_j \geq 0$. 
The trivial bound for the sup-norm of the Laplace eigenfunction $ \phi_j $ on $ \GmodHn $ is 
\begin{equation} \label{convexitysupnorm} 
\|\phi_j\|_{\infty} \ll t_j^{\frac{n-1}{2}}. 
\end{equation}

\noindent 
For $n \geq 5$ Donnelly proved: 

\begin{theorem}[Donelly \cite{donnelly}] \label{donnellyresult} 
For every $ n \geq 5 $ there exist compact arithmetic manifolds 
$ \Gamma \setminus \mathbb{H}^n $ admitting sequences of eigenfunctions 
satisfying 
\begin{equation} \label{donnellybound} 
\|\phi_j\|_{\infty}  = \Omega \left(t_j^{\frac{n-4}{2}}\right) 
\quad j \rightarrow \infty. 
\end{equation} 
\end{theorem}

\noindent 
The sequences $ (\phi_j)_j $ that satisfy (\ref{donnellybound}) are 
called exceptional sequences. Furthermore, it follows from the work 
of Brumley and Marshall~\cite{brumley} that there also exist $ 4 $-manifolds 
which admit eigenfunctions with large sup-norm: they obtain the bound 
\begin{equation} \label{brumleymarshall}
\|\phi_j\|_{\infty} = \Omega\left(t_j^{b_4}\right)
\end{equation}
for some $ b_4> 0 $. In fact, similarly to \cite{milicevic} and \cite{rudnicksarnak},
 in \cite{brumley} and \cite{donnelly}
the following stronger statements are proved: we can find thin 
subsequences $ (\lambda_{j_k})_{j_k} $ and fixed arithmetic points $ P \in \GmodHn $ 
such that, as $ j_k \rightarrow \infty $, we have 
\begin{equation} \label{fixedpointlowerbound}
|\phi_{j_k}(P)| \gg 
\begin{cases}
t_{j_k}^{b_4} & \mbox{if } n = 4, \\
 t_{j_k}^{(n-4)/2} & \mbox{if } n \geq 5.
\end{cases} 
\end{equation}
\noindent 
Thus we can prove our next result. 

\begin{theorem} \label{theoremappendix} 
Let $ \GmodHn $ be a hyperbolic $ n $-manifold satisfying (\ref{donnellybound}) 
or (\ref{brumleymarshall}), respectively and $ P $ a fixed arithmetic 
point as in (\ref{fixedpointlowerbound}). Furthermore, choose $ \delta_n $ such that 
\begin{equation} \label{exponentsdelta}
\delta_n = 
\begin{cases} 
1 - \frac{2 b_4}{5} & \mbox{if } n = 4, \\ 
\frac{5}{n+1} & \mbox{if } n \geq 5. 
\end{cases} 
\end{equation}
If $ R \ll t_j^{-\delta_n +\epsilon} $, then
\begin{equation*}
\frac{1}{\vol(B_R(P))} \int_{B_R(P)} |\phi_j(Q)|^2 dv(Q) 
= 
\Omega_{\epsilon}\left(t_j^{\epsilon}\right).
\end{equation*}
In particular, QUE on shrinking balls fails for radii away from the 
Planck-scale.
\end{theorem}

\noindent 
\begin{remark} 
We mention that in \cite{brumley} no explicit value for $ b_4 $ is 
given. However, by (\ref{convexitysupnorm}) we get $ b_4 \leq 3/2 $ 
and this allows to bound the exponent $ \delta_4 $ appearing in 
Theorem~\ref{theoremappendix}. Namely, we have $ 2/5 \leq \delta_4
< 1 $. 
\end{remark}

\noindent 
The proof of Theorem~\ref{theoremappendix} relies on the following lemma, 
which generalizes the third case of \cite[Lemma~4.2]{humphries} and provides 
also a modified generalization of the third case in \cite[Lemma~2.4]{chamizo}. 

\begin{lemma} \label{lemmandimensionsshc} 
For $ n \geq 3 $ there exists a constant $ c_n $ such that, as $ R 
\rightarrow 0 $ and $ Rt \rightarrow \infty $, the Selberg transform 
$ h_{R,n}(t) $ associated to the characteristic function via (\ref{kkernel}) 
satisfies 
\begin{equation*}
\begin{split}
h_{R,n}(t) 
&= 
\frac{2^{\frac{n+1}{2}}}{\sqrt{\pi}} \, \Gamma\left(\frac{n}{2}+1\right) 
\left(Rt\right)^{-\frac{n+1}{2}} \cos\left(Rt - \frac{(n+1) \pi}{4}\right) \\ 
& \quad 
+ O_{n,M} \left((tR)^{-\frac{n+3}{2}} + R^2 (tR)^{-\frac{n+1}{2}} + R^{2M+2}\right)     
\end{split}
\end{equation*}
for any fixed integer $ M \geq 1 $. 
\end{lemma}

\begin{proof} 
By \cite[Eq.~(4.2)]{phirud} and \cite{laxphillips} the Selberg transform of 
the kernel $ k(u) $ defined by (\ref{kkernel}) is given by 
\begin{equation} \label{selbergharishchandra2}  
\begin{split} 
h_{R,n}(t) 
&=
\frac{a_n}{\vol(B_R)} \int_{0}^{R} \left(\cosh R - \cosh u\right)^{\frac{n-1}{2}} 
\cos(tu) \, du \\ 
&= 
\frac{a_n \, R}{\vol(B_R)} \int_0^1 \left(\cosh R - \cosh Ru\right)^{\frac{n-1}{2}} 
\cos(tRu) \, du 
\end{split} 
\end{equation} 
where the constant 
\begin{equation*}
a_n = 
\frac{2^{\frac{n+3}{2}} \pi^{\frac{n-1}{2}}}{(n-1) \, \Gamma\left(\frac{n-1}{2}\right)}   
\end{equation*}
depends only on the dimension $ n $. In order to determine the asymptotic behaviour of 
$ h_{R, n} $ we distinguish between the case $ n $ even and $ n $ odd. If $ n $ is even, 
we get using the Taylor expansion of the hyperbolic cosine 
and 
of $ (1 + x)^{\frac{n-1}{2}} $ 
\begin{equation*}
\begin{split}
& \left(\cosh R - \cosh Ru\right)^{\frac{n-1}{2}}  
= 
\left(\frac{R^2 \left(1 - u^2\right)}{2}\right)^{\frac{n-1}{2}} 
\left(1 + 2 \sum_{k = 1}^M \frac{R^{2k}}{(2k)!} \sum_{l=0}^k u^{2l} 
+ O\left(R^{2M+2}\right)\right)^{\frac{n-1}{2}} \\ 
& \qquad = 
\left(\frac{R^2 \left(1 - u^2\right)}{2}\right)^{\frac{n-1}{2}} 
\left(\sum_{r = 0}^N 
\left(\begin{array}{c} 
\frac{n-1}{2} \\ r 
\end{array}\right) x^r + O_N\left(x^{N+1}\right)\right)  \\ 
& \qquad = 
\left(\frac{R^2 \left(1 - u^2\right)}{2}\right)^{\frac{n-1}{2}} 
\left(1 + \sum_{r = 1}^{N} \sum_{k = r}^{M} \sum_{l = 0}^k c_{k,l,r} R^{2k} u^{2l} 
+ O_M\left(R^{2M+2}\right) + O_N\left(R^{2N+2}\right)\right) 
\end{split}
\end{equation*}
with certain positive coefficients $ c_{k,l,r} $ and $ N \leq M $. If $ n $ is odd, 
we can avoid the Taylor expansion of $ (1+x)^{\frac{n-1}{2}} $ and obtain immediately 
\begin{equation*}
\left(\cosh R - \cosh Ru\right)^{\frac{n-1}{2}}  
= 
\left(\frac{R^2 \left(1 - u^2\right)}{2}\right)^{\frac{n-1}{2}} 
\left(1 + \sum_{r = 1}^{N} \sum_{k = r}^{M} \sum_{l = 0}^k c_{k,l,r} R^{2k} u^{2l} 
+ O_M\left(R^{2M+2}\right)\right).  
\end{equation*}
By (\ref{selbergharishchandra2}) this implies that 
\begin{equation*} 
h_{R, n}(t) 
= 
MT_1 + MT_2 + O_{n,M}\left(R^{2M+2}\right) 
\end{equation*}
where the main term is given by 
\begin{align}
\label{MT1} 
MT_1 &= \frac{a_n \, R^n}{2^{\frac{n-1}{2}} \, \vol(B_R)} \int_0^1 
(1 - u^2)^{\frac{n-1}{2}}  \cos(tRu) \, du \\ 
\intertext{and} 
\label{MT2} 
MT_2 &= 
\frac{a_n \, R^n}{2^{\frac{n-1}{2}} \, \vol(B_R)} \int_0^1 \sum_{r = 1}^{N} 
\sum_{k = r}^{M} \sum_{l = 0}^k c_{k,l,r} R^{2k} u^{2l} \, 
(1 - u^2)^{\frac{n-1}{2}}  \cos(tRu) \, du. 
\end{align} 
Note that for small radii $ R $ the hyperbolic volume of $ B_R $ is 
asymptotic to the Euclidean volume, i.e.\ $ \vol(B_R) = \pi^{\frac{n}{2}} 
R^n / \Gamma\big(\frac{n}{2} + 1\big) + O\left(R^{n+2}\right)$ as $ R \to 0 $. 
It remains to determine the main term $ MT_1 $ and the contribution of the 
second error term $ MT_2 $. Using \cite[Eq.~(17.34.10)]{gradry} we get for 
the main term 
\begin{equation} \label{asympfirstterm}
\frac{a_n \, R^n}{2^{\frac{n-1}{2}} \, \vol(B_R)} 
\int_0^1 \left(1-u^2\right)^{\frac{n-1}{2}} \cos(tRu) \, du = 
\frac{a_n \sqrt{\pi} \, R^n}{\sqrt{2} \, \vol(B_R)} \Gamma\Big(\frac{n+1}{2}\Big) 
\frac{J_{\frac{n}{2}}(Rt)}{(Rt)^{\frac{n}{2}}} 
\end{equation}
where $ J_n(x) $ denotes the $ J $-Bessel function. As the asymptotic behaviour 
of the Bessel function is given by 
\begin{equation} \label{eq:Bessel} 
J_{n/2}(Rt) 
= \sqrt{\frac{2}{\pi R t}}
\cos\left(Rt - \frac{(n+1) \pi}{4}\right) + O_n\left((Rt)^{-3/2}\right) 
\end{equation}
(see e.g.\ \cite[(B.35)]{iwaniec}), we infer that 
\begin{equation*}
\begin{split}
MT_1 
&= 
\frac{2^{\frac{n+1}{2}}}{\sqrt{\pi}} \Gamma\left(\frac{n}{2}+1\right) 
(Rt)^{-\frac{n+1}{2}} \cos\left(Rt - \frac{(n+1) \pi}{4}\right) 
+ O_n\left(R^2 (Rt)^{-\frac{n+1}{2}}\right) \\ 
& \quad 
+ O_n\left((Rt)^{-\frac{n+3}{2}}\right). 
\end{split}
\end{equation*}
In order to treat $ MT_2 $ we consider the integral 
\begin{equation*}
I_l := \int_0^1 \left(1-u^2\right)^{\frac{n-1}{2}} u^{2l} \, \cos(tRu) \, du, 
\end{equation*}
$ l \in \N $, that appears in $ MT_2 $. By \cite[Eq.~(17.34.10)]{gradry} and 
\cite[Eq.~(8.471.2)]{gradry} we obtain that 
\begin{equation*}
\begin{split}
I_l 
&= 
\frac{(-1)^l}{R^{2l}} \frac{\partial^{2l}}{\partial t^{2l}} \int_0^1 
\left(1-u^2\right)^{\frac{n-1}{2}} \cos(tRu) \, du \\[1 ex]  
&= 
\frac{(-1)^l 2^{\frac{n}{2} - 1} \sqrt{\pi}}{R^{2l}} \, \Gamma\Big(\frac{n+1}{2}\Big) 
\frac{\partial^{2l}}{\partial t^{2l}} \, \frac{J_{\frac{n}{2}}(Rt)}{(Rt)^{\frac{n}{2}}} \\ 
&= 
(-1)^l 2^{\frac{n}{2} - 1} \sqrt{\pi} \, \Gamma\Big(\frac{n+1}{2}\Big) \sum_{k = 0}^{2l} 
\alpha_k \left((Rt)^{-\frac{n}{2}}\right)^{(k)} \left(J_{\frac{n}{2}}(Rt)\right)^{(2l-k)} \\  
&= 
(-1)^l 2^{\frac{n}{2} - 1} \sqrt{\pi} \, \Gamma\Big(\frac{n+1}{2}\Big) \sum_{k = 0}^{2l} 
\alpha_{k,n} (Rt)^{-\frac{n}{2} - k} \sum_{m=-(2l-k)}^{2l-k} \beta_m J_{\frac{n}{2}+m}(Rt) \\ 
\end{split}
\end{equation*}
holds with certain coefficients $ \alpha_k, \alpha_{k,n}, \beta_m \in \R $. Thus 
(\ref{MT2}) and (\ref{eq:Bessel}) imply 
\begin{equation*}
MT_2 = O\left(R^2 (Rt)^{-\frac{n+1}{2}}\right). 
\end{equation*}
This proves the lemma. 
\end{proof}

\noindent  
What is crucial in the proof of Lemma (\ref{lemmandimensionsshc}) is that we have 
freedom in the choice of $ M $. If $ M $ is chosen such that 
$ 
R \ll t^{-\frac{n+1}{4(M+1)+n+1}} 
$, then 
\begin{equation*}
R^{2M+2}  \ll \left(Rt\right)^{-\frac{n+1}{2}} 
\end{equation*}
Taking $ M $ sufficiently large we infer that our main term is a good approximation 
of the Selberg transform $ h_{R,n} $ in the range $ R \asymp t^{-\delta} $ with 
$ \delta \in (0, 1) $. Hence 
\begin{equation*}
h_{R,n}(t) = \Omega \left(\left(Rt\right)^{-\frac{n+1}{2}}\right). 
\end{equation*} 

\noindent 
We now prove Theorem~\ref{theoremappendix}. 

\begin{proof}[Proof of Theorem~\ref{theoremappendix}] 
First of all, we note that the $n$-dimensional analogue of Proposition~\ref{meanvalueprop} 
is a standard property of the Selberg transform \cite[eq.~(1.8)]{selberg}. 
Then, as in the proof of Theorem~\ref{proposition1} 
the Cauchy--Schwarz inequality implies 
\begin{equation*}
\frac{1}{\vol(B_R(P))} \int_{B_R(P)} |\phi_j(Q)|^2 dv(Q) 
\geq |h_{R,n}(t_j)|^2 |\phi_j(P)|^2.
\end{equation*}
Now let us consider the exceptional subsequence $ (\lambda_{j_k})_k $. For 
$ n = 4 $ Lemma~\ref{lemmandimensionsshc} and (\ref{brumleymarshall}) imply 
that 
\begin{equation*}
\frac{1}{\vol(B_R(P))} \int_{B_R(P)} |\phi_{j_k}(Q)|^2 dv(Q) 
\gg 
t_{j_k}^{2b_4 - 5(1-\delta_4)}
\end{equation*}
if $ R \ll t_{j_k}^{-\delta_4} $ and $ t_{j_k}^{2b_4 - 5(1-\delta_4)} 
\rightarrow \infty $ if $ \delta_4 > 1 - \tfrac{2 b_4}{5} $. For $ n \geq 5 $ 
we replace (\ref{brumleymarshall}) by Theorem~\ref{donnellyresult} and 
obtain 
\begin{equation*}
\frac{1}{\vol(B_R(P))} \int_{B_R(P)} |\phi_{j_k}(Q)|^2 dv(Q) 
\gg 
\left(t_{j_k}^{-1 + \delta_n}\right)^{2 \left(\frac{n+1}{2}\right)} 
\left(t_{j_k}^{\frac{n-4}{2}}\right)^2 
\gg 
t_{j_k}^{\delta_n (n+1) - 5} 
\end{equation*}
if $ R \ll t_{j_k}^{-\delta_n} $. This proves the theorem as 
$ t_{j_k}^{\delta_n (n+1) - 5} \rightarrow \infty $ if $ \delta_n > 
\frac{5}{n+1} $. 
\end{proof}

\begin{remark}
It is natural to expect that for $ n \geq 5 $ the exponent $ \delta_n = 
\frac{5}{n+1} $ is optimal, in the same sense that we expect $ \delta = 
3/4 $ to be optimal for arithmetic $ 3 $-manifolds.
\end{remark}

\begin{remark}
One can also consider the more general case where the centre $ w_j $ 
of the shrinking balls $ B_{R_j}(w_j) $ varies. For the situation of 
the flat torus this more general case has been treated by 
Lester--Rudnick~\cite{lesterrudnick}. However, in the hyperbolic case, 
even for in dimension $ 2 $, the assumption that the centre $ w $ is 
fixed or belongs to a compact set is necessary. For the hyperbolic 
plane and the modular group $ \G = \pslz $ Iwaniec and Sarnak 
proved that 
\begin{equation*}
t_j^{\frac{1}{6}-\epsilon} 
\ll 
u_j\left(w_j\right) 
\ll 
t_j^{\frac{1}{6}+\epsilon}
\end{equation*}
if $ w_j = \frac{1}{4} + i \, \frac{t_j}{2 \pi} + o(1) $ (see 
\cite{Sarnak-letter-Morawetz}). 
This purely analytic phenomenon implies that 
\begin{equation*}
\frac{1}{\vol(B_R(w_j))} \int_{B_R(w_j)} \left|u_j(z)\right|^2 d\mu(z) 
\gg 
|h(t_j)|^2 \left|u_j(w_j)\right|^2 
\gg 
t_j^{-\frac{8}{3} -\epsilon} R^{-3} 
\end{equation*}
which is unbounded for any $ R \ll t_j^{-\delta}$ with $ \delta > 8/9 $ 
fixed.
\end{remark}

It follows from the work of Phillips and Sarnak~\cite{phisarnak} that 
a generic subgroup $ \G \subset \psl $ has few cusp forms, conjecturally 
there are only finitely many. This is a consequence of the contribution 
of the continuous spectrum to the Weyl law. 
In general, if a co-finite group $ \Gamma \subset \textup{SO}(n,1) $ has 
sufficiently large Eisenstein series in the sense that
\begin{equation} \label{largeesinsteincond}
\sum_{\mathfrak{a}} \int_{-T}^{T} \left|E_{\mathfrak{a}}\left(P, 
\frac{n-1}{2} + it \right)\right|^2 dt \gg_P T^n 
\end{equation}
for some point $ P \in \G \setminus \mathbb{H}^n $, then
\begin{equation*}
E_{\mathfrak{a}}\left(P, \frac{n-1}{2} + it\right) 
= 
\Omega\left(t^{\frac{n-1}{2}}\right). 
\end{equation*}
Here $ E_{\mathfrak{a}}(P, s) $ denotes the Eisenstein series associated 
to the cusp $ \mathfrak{a} $. Note that the modular group $ \textup{PSL}_2(\Z) $ 
does not satisfy condition (\ref{largeesinsteincond}). Namely, in this case 
the following bound holds for $ z $ fixed: 
\begin{equation*} 
E(z, 1/2 + it) \ll t^{1/3 + \epsilon} 
\end{equation*}
(see \cite{blomer}, \cite{nordentoft}). For groups that satisfy 
(\ref{largeesinsteincond}) we obtain the following proposition 
working as in Theorem~\ref{theoremappendix}: 

\begin{proposition}\label{propositionappendix}
If $ \Gamma $ has sufficiently large Eisenstein series in the sense of 
(\ref{largeesinsteincond}), then QUE of Eisenstein series 
fails in shrinking balls of radius $ R \ll t^{-\frac{2}{n+1}+\epsilon} $.
\end{proposition}

Comparing Theorem~\ref{theoremappendix} and Proposition~\ref{propositionappendix} 
to the bound of Han and Hezari--Rivi\`ere (\ref{hanhezriv}) it is natural to ask 
whether for arithmetic hyperbolic manifolds $ \mathcal{M} = \Gamma \setminus \mathbb{H}^n $ 
their bound can be improved. If $ \mathcal{M} = \Gamma \setminus \mathbb{H}^n $ is 
an $n$-dimensional hyperbolic manifold of finite volume, we expect quantum unique 
ergodicity for shrinking balls of radii
\begin{equation*} 
R \gg t_j^{-\frac{c}{n+1} + \epsilon}
\end{equation*}
for some constant $ c = c_{\mathcal{M}} > 0 $ that depends on the manifold 
$\mathcal{M}$ and which might be different for the discrete and the continuous spectrum.

\end{document}